\documentclass[11pt]{amsart}
\usepackage[dvipsnames]{xcolor}

\usepackage{fullpage}

\usepackage{amsmath}
\usepackage{amsfonts}
\usepackage{amssymb}
\usepackage{amsthm}
\usepackage{mathtools} 
\usepackage{latexsym}

\usepackage{physics}

\usepackage{enumerate}
\usepackage{cancel}
\usepackage{cases}
\usepackage{empheq}
\usepackage{multicol}

\usepackage{graphics} 

\usepackage{wrapfig}

\usepackage{color}

\usepackage[backgroundcolor=gray!30,linecolor=black]{todonotes}
   
\usepackage{mathrsfs}
\usepackage{fontenc} 
\usepackage{inputenc} 

\usepackage{verbatim}

\theoremstyle{plain}
\newtheorem{theorem}{Theorem}[section]

\newtheorem{lemma}[theorem]{Lemma}

\theoremstyle{definition}

\theoremstyle{remark}
\newtheorem{remark}[theorem]{Remark}

\numberwithin{equation}{section} 
\numberwithin{figure}{section}   

\usepackage[square,comma,numbers,sort&compress]{natbib}
\usepackage[colorlinks=true, pdfborder={ 00 0}]{hyperref}
\hypersetup{urlcolor=blue, citecolor=red}
\usepackage{url}

\newcommand{\vect}[1]{\mathbf{#1}} 

\newcommand{\bk}{\vect{k}}

\newcommand{\bu}{\vect{u}}

\newcommand{\bv}{\vect{v}}

\newcommand{\bw}{\vect{w}}

\newcommand{\bx}{\vect{x}}

\newcommand{\bbf}{\vect{f}}

\newcommand{\field}[1]{\mathbb{#1}}
\newcommand{\nN}{\field{N}}
\newcommand{\nZ}{\field{Z}}

\newcommand{\nR}{\field{R}}

\newcommand{\bphi}{\vect{\phi}}

\newcommand{\nT}{\mathbb T}

\newcommand{\pd}[2]{\frac{\partial #1}{\partial #2}}
\renewcommand{\abs}[1]{\left\lvert#1\right\rvert}
\newcommand{\set}[1]{\left\{#1\right\}}

\renewcommand{\norm}[1]{\left\|#1\right\|}

\newcommand{\fournorm}[1]{\|#1\|_{L^4}}

\newcounter{my_counter}
\renewcommand{\grad}{\nabla}
\renewcommand{\div}{\nabla \cdot} 
\newcommand{\lap}{\Delta} 

\usepackage{placeins} 

\usepackage{cleveref}

\newcommand{\bvn}{\vect{v_{n}}}
\newcommand{\bwn}{\vect{w_{n}}}

\newcommand{\weaklyto}{\rightharpoonup}
\newcommand{\weakstarto}{\overset{*}{\rightharpoonup}}

\newcommand{\Calpha}{{ C_\alpha }}
\newcommand{\Cnu}{{ C_\nu }}

\usepackage{hyperref}

\title{Continuous Data Assimilation for the 3D and Higher-Dimensional Navier--Stokes equations with Higher-Order Fractional Diffusion}
\date{\today} 


\author{Adam Larios}
\address[Adam Larios]{Department of Mathematics, 
                University of Nebraska--Lincoln,
        Lincoln, NE 68588-0130, USA}
\email[Adam Larios]{alarios@unl.edu}

\author{Collin Victor}
\address[Collin Victor]{Department of Mathematics, 
                University of Nebraska--Lincoln,
        Lincoln, NE 68588-0130, USA}
\email[Collin Victor]{collin.victor@huskers.unl.edu}

\title{Continuous Data Assimilation for the 3D and Higher-Dimensional Navier--Stokes equations with Higher-Order Fractional Diffusion}
\date{\today} 


\allowdisplaybreaks
\begin{document}

\begin{abstract}
    We study the use of the Azouani-Olson-Titi (AOT) continuous data assimilation algorithm to recover solutions of the Navier--Stokes equations modified to have higher-order fractional diffusion. The fractional diffusion case is of particular interest, as it is known to be globally well-posed for sufficiently large diffusion exponent $\alpha$. In this work, we prove that the assimilation equations are globally well-posed, and we demonstrate that the solutions produced by the AOT algorithm exhibit exponential convergence in time to the reference solution, given a sufficiently high spatial resolution of observations and a sufficiently large nudging parameter.

    We also note that the results hold in spatial dimensions $d$ where $2\leq d\leq 8$, so long as $\alpha\geq \frac12 +\frac{d}{4}$. Though the cases $3<d\leq8$ are likely only a mathematical curiosity, we include them as they cause no additional difficulty in the proof.  Note that we show in a companion paper the $d=2$ case allows for $\alpha<1$.
\end{abstract}

\maketitle
\thispagestyle{empty}

\section{Introduction}
\noindent
There is a general understanding in the community that diffusion is needed for data assimilation schemes to converge to the true solution. Just how much diffusion does one need?  In the present work, we investigate this question in the context of the 3D Navier--Stokes equations (NSE) with fractional diffusion.  In particular, we prove that Azouani-Olson-Titi (AOT)-style data assimilation converges exponentially fast in time to the reference solution in the 3D case if one uses fractional diffusion with an exponent ($\alpha$) that is at least the Lions' exponent ($\alpha\geq5/4$), i.e., slightly more than the physical viscosity.  
Given that the large-time behavior of the 3D Navier--Stokes equations poses well-known major challenges for conventional analysis, we examine this variant of the NSE to enhance understanding of the algorithm's performance in the 3D setting.
Interestingly (from a purely mathematical perspective), we show that our results extend up to dimension eight, so long as sufficient diffusion is used.  The restriction on the dimension is discussed in \Cref{rmk_dimension_restriction}.    
In a follow-up paper \cite{Carlson_Larios_Victor_2023_2D_fractional_AOT}, we show that the AOT algorithm also enjoys convergence for the 2D NSE with lower-order diffusion ($\alpha<1$); i.e., the case of hypo-diffusion, and also investigate this setting computationally. Note that in the 2D setting, one can consider an approach that is currently out of reach for the 3D case; namely one can take the reference solution to be the unmodified Navier-Stokes (or even Euler) equations, although the error only converges up to a non-zero level determined by $\alpha$.  This is investigated in \cite{Carlson_Larios_Victor_2023_2D_fractional_AOT}.




The term \textit{data assimilation} refers to a wide class of schemes that combine observational data with the underlying physical model in order to better simulate the future state of a dynamical system. 
Some classical methods of data assimilation are the Kalman filter and its variants, which are a form of linear quadratic estimation, as well as variational methods, such as 3D/4D Var, which come from control theory.  For a deeper discussion of these methods, see, e.g., \cite{asch2016DA,Law_Stuart_Zygalakis_2015_book} and the references therein.
Instead of these classical methods, in this work we investigate the usage of the Azounai-Olson-Titi (AOT) continuous data assimilation algorithm, which has many advanges, and has seen much activity in recent years, as discussed below.

The AOT algorithm differs from other contemporary methods of data assimilation by implementing a feedback-control term at the partial differential equation (PDE) level. The algorithm begins with a dynamical system, e.g., the incompressible NSE, given abstractly\footnote{We denote $\bu_t:=\pd{\bu}{t}$ and we denote initial data by $\bu_0$, but these subscripts should not be a source of confusion.} as 
\begin{align}\label{eq: F}
    \bu_t = F(\bu).
\end{align}
Note that the initial state of the system, call it $\bu_0$, may be entirely unknown for the purposes of applying the AOT algorithm. 
If the initial state of system \cref{eq: F} was known exactly, at every point in the domain, then in principle, the future evolution of the solution could be determined without using data assimilation.  However, in practice, the information about the initial data might only be known sparsely in space (e.g., at the locations of weather monitoring devices).  Given a sufficiently chaotic system (e.g., NSE in a turbulent regime), these unknowns can rapidly create large-scale errors, resulting in drastically inaccurate predictions.  
To resolve this issue, the AOT algorithm feeds spatially sparse observational data of the ``true'' solution $\bu$ into a model via the following PDE, equipped with arbitrary admissible initial data:
\begin{align}\label{eq: F AOT}
    &\bv_t = F(\bv) + \mu I_h(\bu- \bv),\\
    &\bv(\bx,0) = \bv_0. \notag
\end{align}
Here $\mu > 0$ is a feedback-control parameter with units inversely proportional to time, and $I_h$ is a linear interpolant, characterized by some ``typical'' length scale $h>0$ (e.g., $h$ might correspond to some typical distance between observation locations). We note that the AOT algorithm appears similar to nudging, also known as Newtonian relaxation proposed in \cite{Anthes_1974_JAS,Hoke_Anthes_1976_MWR}. This similarity is superficial as the interpolation operator in \cref{eq: F AOT} has a large impact on the implementation and convergence rates for this algorithm that separates it from nudging methods. Moreover, the AOT algorithm is supported by rigorous mathematical analysis (discussed below), unlike the case of Newtonian relaxation.  For an in-depth examination of nudging methods see, e.g., \cite{lakshmivarahan2013nudging} and the references contained within.

The AOT algorithm was first developed in \cite{Azouani_Titi_2014} for a 1D reaction diffusion equation and then given in \cite{Azouani_Olson_Titi_2014} for general classes of interpolants in the context of the 2D incompressible Navier--Stokes equations. The first simulations of the algorithm were performed in \cite{Gesho_Olson_Titi_2015}. Subsequently the AOT algorithm has been the subject of much study since its development in 2014. 
While the algorithm was first applied to the 1D Allen--Cahn equation and 2D NSE in \cite{Azouani_Olson_Titi_2014,Azouani_Titi_2014}, it has since been applied successfully to many other dissipative dynamical systems. This includes the Cahn--Hilliard equations \cite{Diegel_Rebholz_2021}, the 3D primitive ocean equations \cite{Carlson_VanRoekel_Petersen_Godinez_Larios_2021,Larios_Petersen_Victor_2023,Pei_2019}, the surface quasi-geostrophic equations \cite{Jolly_Martinez_Olson_Titi_2018_blurred_SQG,Jolly_Martinez_Titi_2017}, the Kuramoto--Sivashinsky equation \cite{Larios_Pei_2017_KSE_DA_NL,Lunasin_Titi_2015,Pachev_Whitehead_McQuarrie_2021concurrent}, and various regularizations of NSE \cite{Albanez_Nussenzveig_Lopes_Titi_2016,Gardner_Larios_Rebholz_Vargun_Zerfas_2020_VVDA,Larios_Pei_2018_NSV_DA}.  One could include the present work in this last category as well, as higher-order diffusion can be thought of as a regularization of the NSE. There have also been several studies that investigate the ability of this algorithm to recover the solutions to dynamical systems with a variety of restrictions placed on the observational data. This includes observations that include measurement error \cite{Bessaih_Olson_Titi_2015,Celik_Olson_2022,Foias_Mondaini_Titi_2016}, incomplete observations of each variable \cite{Farhat_Jolly_Titi_2015,Farhat_Lunasin_Titi_2016abridged,Farhat_Lunasin_Titi_2016benard,Farhat_Lunasin_Titi_2016_Charney,Farhat_Lunasin_Titi_2017_Horizontal}, observations that are discrete or averaged in time \cite{Celik_Olson_2022,Foias_Mondaini_Titi_2016,Jolly_Martinez_Olson_Titi_2018_blurred_SQG,Larios_Pei_Victor_2023}. There have also been numerous studies adjusting the assimilation of the observational data, including assimilating observational data on a localized patch of the domain \cite{Biswas_Bradshaw_Jolly_2020}, using a mobile observers to gather data \cite{Biswas_Bradshaw_Jolly_2022,Franz_Larios_Victor_2022,Larios_Victor_2019}, and using nonlinear variations of the feedback-control term 
\cite{Carlson_Larios_2020_nonlinConv,Larios_Pei_2017_KSE_DA_NL}. 
Related algorithms which use interpolated or filtered data as in AOT, but insert the results directly have been studied in \cite{Celik_Olson_Titi_2019,Hayden_Olson_Titi_2011,Larios_Pei_Victor_2023,Olson_Titi_2003}.
Recently there has been work done on modifying this algorithm for use not only in recovering solutions, but in system identification. That is, the AOT algorithm has been modified to simultaneously recover the true solution of a dynamical system along with the viscosity \cite{Carlson_Hudson_Larios_2020, Farhat_GlattHoltz_Martinez_McQuarrie_Whitehead_2019} or the forcing \cite{farhat2023identifying,Martinez_2022}.


The motivation for adding hyper-diffusion is to dampen the small scales generated by the nonlinearity by intensifying the stabilizing effect of the diffusive term.  
In the context of data assimilation, we show in the present work that this means that not only is the PDE model globally well-posed, but the small scales are sufficiently subordinate to the large scales to the extent that they may be captured (asymptotically in time) by observing only the large scales, at least when $\alpha\geq \frac{d}{4}+\frac12$.  It is worth mentioning that non-uniqueness of weak solutions for the 3D NSE with fractional dissipation and $\alpha<5/4$ was shown in 
\cite{Luo_Titi_2020_Lions} (this was first carried out in the $\alpha=1$ case in \cite{Buckmaster_Vicol_2019_nonuniqueness}).  While it is unclear exactly how connects to the results shown here, the results of \cite{Luo_Titi_2020_Lions} indicate that the constraint in this work that $\alpha\geq5/4$ (when $d=3$) may be a fundamental barrier, and not just a technical one.

This work is organized as follows. In \Cref{preliminaries} we outline the formulation of the PDEs, the functional setting of the Navier--Stokes equations, as well as some useful inequalities used in the remainder of the work. In \Cref{well posedness} we prove global well-posedness for classical solutions to \cref{eq: F AOT} by applying Galerkin arguments as well as a bootstrap argument to obtain higher regularity. In \Cref{convergence} we prove our main results, including exponential convergence of $\bv $ to $\bu$ in relevant Sobolev spaces, $H$ and $V^\alpha$, for any admissible choice of initial data $\bv_0$. In \Cref{conclusion} we give some concluding remarks.  For part of this work, we require uniform bounds on some of the quantities.  It seems that such results, at least for the $d>3$ and $\alpha\geq \frac{d}{4}+\frac12$ case, are somewhat obscured in the literature.  Hence, for the convenience of the reader and also for completeness, we prove these results in the appendix.

\section{Preliminaries}
\label{preliminaries}
\noindent
We consider the $d$-dimensional incompressible Navier--Stokes equations with fractional diffusion:
\begin{align}
    \bu_t + \bu \cdot \nabla \bu + \nu (-\lap)^\alpha \bu &= \nabla p +\vect{f},\label{eq: NSE}
    \\
    \div \bu &= 0. \notag
\end{align}
Here $\bu(\bx,t) = (u_1(\bx, t), u_2(\bx,t),..., u_d(\bx,t))$ is the velocity, $\nu>0$ is the viscosity, $p$ is the pressure, and $\vect{f}$ is a forcing term. We will consider these equations in dimensions $2 \leq d \leq 8$, as there is no additional difficulty in examining these higher dimensions. 

When we apply the AOT algorithm to these equations, we obtain
\begin{align}
        \bv_t + \bv \cdot \nabla \bv + \nu (-\lap)^\alpha \bv &= \nabla q +\vect{f}+ \mu I_h(\bu - \bv),\label{eq: NSE AOT}
        \\
    \div \bv &= 0. \notag
\end{align}
Here $I_h$ is a linear interpolation term with associated length scale $h$, and $\mu >0$ is a feedback-control parameter.
We note that, for our main results to hold, we require $I_h$ to satisfy the following inequality for any $s\geq 0$.
\begin{align}
    \norm{I_h(\vect{f}) - \vect{f}}_{V^s}^2 \leq c_{0,\alpha} h^2\norm{\vect{f}}_{V^{s+\alpha}}^2.\label{eq: interp}
\end{align}
The exact definition of these norms will be discussed below.

\subsection{Functional Setting of Navier--Stokes equations}
To begin, we denote the following space of test functions 
\begin{align}
    \mathcal{V} := \left\{ \vect\varphi \in \mathcal{F}: \grad \cdot \vect\varphi = 0 \text{ and } \int_\Omega \vect\varphi(\bx) \,d\bx = 0 \right\}, 
\end{align}
where $\mathcal{F}$ is the set of all d-dimensional vector-valued trigonometric polynomials defined on the domain 
$\Omega = \mathbb{T}^d = \mathbb{R}^d/\mathbb{Z}^d = [0,1]^d$. 
We define $H = \overline{\mathcal{V}}^{L^2}$ and 
$V = \overline{\mathcal{V}}^{V^\alpha}$, where $H$ and $V^\alpha$ are standard Lebesgue and Sobolev spaces defined on $\Omega$. We define the inner products on $H$ and $V$ respectively by
\begin{align}
 (\bu,\bv) = \sum_{i=1}^d \int_\Omega u_iv_id{\bx} \quad\text{ and }\quad ((\bu,\bv)) = \sum_{i,j = 1}^d \int_\Omega \pdv{u_i}{x_j}\pdv{v_i}{x_j}d{\bx},
 \end{align}
which are associated with the norms $\norm{\bu}_{L^2} = \norm{\bu}_{L^2} = (\bu,\bu)^{1/2}$ and $\norm{\bv}_{V} = ((\bu,\bu))^{1/2}$. We note that the $V$ norm is the $H^1$ seminorm, which we use as the full norm for $V$ which is justified by the Poincar\'e inequality.

We now apply the Leray-Helmholtz projector, $P_\sigma$, to \cref{eq: NSE,eq: NSE AOT} to formally obtain the following equations for $\bu$ and $\bv$, respectively
\begin{align}
    \bu_t + B(\bu,\bu) + \nu A^\alpha \bu &=\vect{f},\label{eq: NSE functional}\\ 
    \bv_t + B(\bv,\bv) + \nu A^\alpha \bv &=\vect{f}+ \mu P_\sigma I_h(\bu - \bv).\label{eq: NSE AOT functional} 
\end{align}
Here we formally denote the nonlinear operator, $B(\vect{a},\vect{b}):= P_\sigma (\vect{a} \cdot \grad \vect{b})$, and the Stokes operator $A:= -P_\sigma \lap$. Both of these operators will be discussed in depth below. We assume without loss of generality that $P_\sigma\vect{f}= \vect{f}$, with the understanding that any portion of the forcing that is not divergence free has already been absorbed into the pressure gradient.

The domain of the Stokes operator, $A$, is given by $\mathcal{D}(A) := H^2 \cap V$. $A$ is defined as the extension of the operator $-P_\sigma \lap$  to a linear operator from $V$ to $V'$, associated with the bilinear form 
\begin{align}
    \langle A\bu, \bv\rangle =  ((\bu,\bv)), \text{ for all } \bv\in V.
\end{align}
We note that $A$ has many well-known properties, including the fact that $A$ has non-decreasing eigenvalues $\lambda_k>0$, and subsequently enjoys the following Poincar\'e inequalities:
\begin{align}
    \lambda_1 \norm{\bu}_{L^2}^2 \leq \norm{\bu}_{V}^2 \text{ for } \bu \in V; \qquad \lambda_1\norm{\bu}_{V}^2 \leq \norm{A\bu}_{L^2}^2 \text{ for } \bu \in \mathcal{D}(A).
\end{align}
 For a more in-depth look at the various properties of $A$ see, e.g., \cite{Robinson_2001, Temam_1997_IDDS, Constantin_Foias_1988}.

We denote the fractional Laplacian via its Fourier symbol, i.e., 
\begin{align}
    (-\lap)^\alpha \bu(\bx) := \sum_{\bk\in \nZ^{d}}(2\pi \abs{\bk})^{2\alpha}\widehat{\bu}_\bk e^{2\pi i \bk\cdot\bx} ,
\end{align}
where $\widehat{\bu}_\bk$ denotes wavemode $\bk$ of $\widehat{\bu}$, i.e., the $\bk^{\text{th}}$ Fourier coefficient of $\bu$.  We also denote the Riesz derivative operator $\Lambda:=(-\lap)^{1/2}$.
\begin{remark}\label{As_notation}
    We use the somewhat non-standard notation
    \[
    A^s:=P_\sigma (-\lap)^s
    \]
for $s>0$, which arises naturally after applying $P_\sigma$ to \cref{eq: NSE}.  Note that this is distinct from the more standard meaning $(P_\sigma (-\lap))^s$, but it coincides with the standard notation for the Stokes operator when $s=1$.
\end{remark}
Using this (modified) Stokes operator, we define the fractional Sobolev space $V^s$ as the closure of $\mathcal{V}$ with respect to the inner product
\begin{align}
    (\bu,\bv)_{V^s} = (A^{s/2} \bu, A^{s/2}\bv),
\end{align}
which induces the norm\footnote{Note that, thanks to the Poincar\'e inequality \cref{poincare}, $\norm{\cdot}_{V^s}$ is a norm, rather than merely a semi-norm.} $\norm{\bu}_{V^s} = \norm{A^{s/2} \bu}_{L^2}$.
Thanks to the mean-free condition, we also have the following Poincar\'e inequalities on $\nT^n$:
\begin{align}\label{poincare}
    \lambda_1^\alpha \norm{\bu}_{L^2}^2 \leq \norm{\bu}_{V^\alpha}^2 \text{ and } \lambda_1^{\alpha - 1}\norm{\bu}_{V}^2 \leq \norm{\bu}_{V^\alpha}^2 \text{ for } \bu \in  V^\alpha.
\end{align}
We note that for the above to hold, we require $\alpha \geq 1$.

We use the following integration-by-parts type of calculation in terms of $A^\alpha$ frequently throughout this work.  It is slightly non-standard due to the mixture of the fractional power and the Leray projection operator, so we highlight it here.  Let $\bv,\bw\in V^{2\alpha}$.  Then, since $P_\sigma\bv=\bv$ and $P_\sigma\bw=\bw$,
\begin{align}
    \left\langle A^\alpha \bv, \bw \right\rangle 
    &= \notag
    \left\langle P_\sigma(-\lap)^\alpha \bv, \bw \right\rangle 
    = 
    \left\langle (-\lap)^{\alpha/2} \bv, (-\lap)^{\alpha/2}\bw \right\rangle
    \\&= \notag 
    \left\langle P_\sigma(-\lap)^{\alpha/2} \bv, P_\sigma(-\lap)^{\alpha/2}\bw \right\rangle
    =
    \left\langle A^{\alpha/2}\bv, A^{\alpha/2}\bw \right\rangle,
\end{align}
where we used the symmetry of $P_\sigma$ and fact that $P_\sigma(-\lap)^{\alpha/2}=(-\lap)^{\alpha/2}P_\sigma$ (since both operators are Fourier multipliers).

In the remainder of this work, we write $C$, $C_a$, $C_{a,b}$, etc.,  to denote a positive constant depending on the indicated parameters which may change from line to line. 


The following bilinear term defined for $\bu, \bv \in \mathcal{V}$ by
\begin{align}
    B(\bu,\bv) := P_\sigma((\bu\cdot\nabla) \bv),
\end{align}
can be extended to a continuous map $B:V\times V\to V'$, which is associated with the trilinear form
\begin{align}
   b(\bu,\bv,\bw):=\langle B(\bu,\bv),\bw \rangle_{V^{-1}}.
\end{align}
In this work we will frequently use properties of the bilinear term which are contained for conciseness in the following lemma (see, e.g., \cite{Constantin_Foias_1988,Temam_2001_Th_Num,Foias_Manley_Rosa_Temam_2001} for details).
\begin{lemma}
Let $\bu,\bv$ and $\bw$ be vector fields mapping $\nT^d\rightarrow\nR^d$ for some $d \geq 2$. Then
    \begin{subequations}
      \begin{align}
      \label{symm1}
      b(\bu,\bv,\bw) &= -b(\bu,\bw,\bv), 
      \quad&\forall\;\bu, \bv, \bw\in V,\\
      \label{symm2}
      b(\bu,\bv,\bv) &= 0,
      \quad&\forall\;\bu, \bv, \bw\in V.
    \end{align}
  \end{subequations}
Moreover, for $\alpha\geq \frac{d}{4}+\frac12$ and $\theta:= 1-\frac{d}{4\alpha}$,
    \begin{subequations}
        \begin{alignat}{3}
            &\abs{b(\bu,\bv,\bw)}       \label{B:inf22}
                \leq C 
                    \norm{\bu}_{L^\infty}
                    \norm{\bv}_{V}
                    \norm{\bw}_{L^2}, 
            &&\quad \forall \;\bu \in L^\infty, \bv\in V, \bw \in H, \\
            &\abs{b(\bu,\bv,\bw)}       \label{B:424}
                \leq C 
                    \norm{\bu}_{L^2}^{\theta}
                    \norm{\bu}_{V^\alpha}^{1-\theta}
                    \norm{\bv}_{V^1}
                    \norm{\bw}^{\theta}_{L^2}
                    \norm{\bw}^{1-\theta}_{V^\alpha}, 
            &&\quad \forall \;\bu \in V^\alpha, \bv \in V^\alpha, \bw \in V^\alpha,\\  
            &\abs{b(\bu,\bv,\bw)}       \label{B:442}
                \leq C 
                    \norm{\bu}_{V^\alpha}
                    \norm{\bv}_{V}^{\theta}
                    \norm{\bv}^{1-\theta}_{V^{\alpha + 1}}
                    \norm{\bw}_{L^2}, 
            &&\quad \forall \;\bu \in V^\alpha, \bv\in V^{\alpha + 1}, \bw \in H.
        \end{alignat}
    \end{subequations}

\end{lemma}

We note that in proofs of \cref{B:424,B:442} and throughout the rest of this work we use the following form of the Gagliardo–Nirenberg-Sobolev interpolation inequality 
\begin{align}\label{Ladyzhenskaya}
\fournorm{\bu} \leq C\norm{\bu}_{L^2}^{\theta} \norm{\bu}_{V^\alpha}^{1-\theta},
\quad\text{for all}\,\, \bu\in V^\alpha,
\text{ with }\theta:= 1 - \frac{d}{4\alpha}.
\end{align}
(Note that, throughout the paper, we assume that $\alpha \geq \frac{d}{4} + \frac{1}{2}$, which implies that $1>\theta\geq\frac{2}{2+d}>0$.)




In addition to these properties of the nonlinear term, we also will utilize the following version of the Kato-Ponce inequality, given in \cite{bae2015gevrey}.

\begin{lemma}[Kato-Ponce Inequality \cite{bae2015gevrey}]
For $1 \leq r < \infty$, $p_i,q_i \neq 1$ such that $\frac{1}{r} = \frac{1}{p_1} + \frac{1}{q_1} = \frac{1}{p_2} + \frac{1}{q_2}$, we have
    \begin{align}
        \norm{\Lambda^s(fg)}_{L^r} \label{ineq:Kato-Ponce}
        \leq 
                C 
                \left(
                        \norm{\Lambda^sf}_{L^{p_1}}
                        \norm{g}_{L^{q_1}}
                    +
                        \norm{f}_{L^{p_2}}
                        \norm{\Lambda^sg}_{L^{q_2}}
                \right),
    \end{align}
    where $C$ is a constant depending on $p_i$, $q_i$, $r$ and $\Omega$.
\end{lemma}


Lastly, we state theorems on the global well-posedness of weak and classical solutions. We note that the global well-posedness of weak solutions was first proved in \cite{Lions_1959} (see Remarque 6.11) and is also proved in \cite{wu2003generalized} for the generalized MHD equations. Moreover the proof of these theorems can be derived in this work by simply setting $\mu = 0$ in the proofs in the next section.f
\begin{theorem}[Global well-posedness of weak solutions for NSE with fractional diffusion\cite{Lions_1959,wu2003generalized}]
    If $\bu_0 \in H$, with time-independent forcing $f\in H$ and $\alpha \geq \frac{d}{4} + \frac{1}{2}$ then \cref{eq: NSE functional} admits a weak solution $\bu\in C_w(0,T;H)\cap L^2(0,T;V^\alpha)$ and $\frac{d\bu}{dt} \in L^\frac{4\alpha}{d}(0,T;V^{-\alpha})$ that solves 
\cref{eq: NSE functional} in the sense of $L^\frac{4\alpha}{d}(0,T;V^{-\alpha})$.
\end{theorem}

\begin{theorem}[Global well-posedness of classical solutions for NSE with fractional diffusion \cite{wu2003generalized}]\label{thm - classical Existence of u}
If $\bu_0\in V^{2\alpha}$, with time-independent forcing, $\bbf \in H$ and  $\alpha \geq \frac{d}{4} + \frac{1}{2}$ then \cref{eq: NSE functional} admits a unique classical solution, $\bu \in C([0,T];V^{2\alpha})\cap L^2(0,T;V^{3\alpha})$.
\end{theorem}

\section{Existence and Uniqueness of Classical Solutions}\label{well posedness}
\noindent
Now we move to establish the global well-posedness of solutions to \cref{eq: F AOT}. To do so we will first prove the existence of weak solutions.

\begin{theorem}\label{thm: weak}
For initial data $\bv_0 \in V^\alpha$, time-independent forcing $\bbf\in H$, and $\alpha \geq \frac{d}{4} + \frac{1}{2}$, given any $\bu \in C([0,T];V^{2\alpha})\cap L^2(0,T;V^{3\alpha})$, the system \cref{eq: NSE AOT functional} 
admits a weak solution $\bv\in C_w(0,T;H)\cap L^2(0,T;V^\alpha)$ and $\frac{d\bv}{dt} \in L^\frac{4\alpha}{d}(0,T;V^{-\alpha})$ that solves 
\cref{eq: NSE AOT functional} in the sense of $L^\frac{4\alpha}{d}(0,T;V^{-\alpha})$.
\end{theorem}

\begin{proof}
To prove this we consider the associated Galerkin system:
\begin{align}\label{galerkin system}
    \dv{\bvn}{t} + P_nB(\bvn, \bvn) + \nu A^\alpha \bvn = P_n\vect{f}+ \mu P_n P_\sigma I_h(\bu - \bvn).
\end{align}
Here $P_n: H\to H_n$ is given by $P_n\bv = \sum_{j=1}^n(\bw_j,\bv)\bw_j$, where $\bw_j$ denotes the $j$th eigenfunction of $A$ and $H_n = \text{span}(\bw_1,...,\bw_n)$.
We look for a solution $ \bvn  \in C^1([0,T_n), H_n)$.
We note that the unknowns (the Fourier coefficients) on the right-hand side form a quadratic polynomial.  Hence, according to the Picard-Lindel\"of Theorem, a unique solution exists locally on the time interval $[0,T_n)$, as the Galerkin system reduces to a system of ODEs that is quadratic in the coefficients of $\bw_j$. Let $T_n^{\text{max}}$ be given to be the maximal interval of existence and uniqueness.

Now we prove that $T_n^{\text{max}} = T$ for any $T>0$. To do so, we simply take the inner product of \cref{galerkin system} with $\bvn$. This yields the following equations with $\vect{f_n} := P_n \vect{f}$.
\begin{align}
    \frac{1}{2}\dv{}{t}\norm{\bvn}^2_{L^2} + \nu \norm{A^{\alpha/2} \bvn}_{L^2}^2 = \langle \vect{f_n},\bvn\rangle + \mu (I_h(\bu),\bvn) - \mu (I_h(\bvn),\bvn).
\end{align}

We note that above the nonlinear term vanishes due to \cref{symm2}. We now move to estimate the terms on the right-hand side.
\begin{align}
\langle \vect{f_n},\bvn\rangle &\leq \frac{1}{\mu}\norm{\vect{f_n}}_{L^2}^2 + \frac{\mu}{4}\norm{\bvn}_{L^2}^2
                        \leq \frac{1}{\mu}\norm{\vect{f}}_{L^2}^2 + \frac{\mu}{4}\norm{\bvn}_{L^2}^2. \notag
\end{align}
For the feedback control terms, we estimate
\begin{align}
\mu (I_h(\bu),\bvn) & \leq C\mu\norm{I_h(\bu)}_{L^2}^2 + \frac{\mu}{4}\norm{\bvn}_{L^2}^2,
\end{align}
and also
\begin{align}
    - \mu (I_h(\bvn),\bvn)  &= \mu (\bvn - I_h(\bvn), \bvn) - \mu(\bvn,\bvn)\\
                            &\leq \mu \norm{\bvn - I_h(\bvn)}_{L^2}\norm{\bvn}_{L^2} - \mu \norm{\bvn}^2_{L^2}\notag \\
                            &\leq \frac{\mu}{2} \norm{\bvn - I_h(\bvn)}_{L^2}^2 + \frac{\mu}{2}\norm{\bvn}^2_{L^2} - \mu \norm{\bvn}^2_{L^2}\notag \\
                            &\leq \frac{\mu c_{0,\alpha} h^2}{2} \norm{A^{\alpha/2}\bvn}_{L^2}^2  - \frac{\mu}{2}\norm{\bvn}^2_{L^2}\notag \\\notag
                            &\leq \frac{\nu}{2} \norm{A^{\alpha/2}\bvn}_{L^2}^2  - \frac{\mu}{2}\norm{\bvn}^2_{L^2}.
\end{align}





Since $\bu \in C([0,T];V^{2\alpha})$ there exists an $M_0>0$ such that $\norm{\bu}^2_{L^2}\leq M_0$ for all time $t\in [0,T]$ for any $T>0$. Moreover as $\bbf\in H$ and $I_h$ is a linear operator it follows that there exists an $M\geq M_0$ such that $\frac{\mu}{2}\norm{I_h(\bu)}^2_{L^2} + \frac{1}{\mu}\norm{\vect{f}}^2_{L^2} \leq M$ for all $t\in [0,T]$.
Hence, we obtain,
\begin{align}
    \frac{1}{2}\dv{}{t}\norm{\bvn}^2_{L^2}  + \frac{\nu}{2} \norm{A^{\alpha/2} \bvn}_{L^2}^2 \leq \frac{\mu}{2}\norm{I_h(\bu)}^2_{L^2} + \frac{1}{\mu}\norm{\vect{f}}^2_{L^2} \leq M.
\end{align}

Now we integrate the equation from $0$ to $t$ to obtain:
\begin{align}
    \norm{\bvn(t)}^2_{L^2} + \nu \int_0^t \norm{A^{\alpha/2} \bvn}^2_{L^2}ds &\leq \norm{\bvn(0)}^2_{L^2} + \int_0^t Mds \\ \notag
    &\leq \norm{\bv_0}^2_{L^2} + MT.
\end{align}
We note that the above holds for any choice of $t \in [0,T_n^{\text{max}})$ as $t$ was arbitrary. We note that if $T_{n}^{\text{max}} < \infty$, then there must exists a time $T_n^{\text{max}}$ such that 
\begin{align}
\lim_{t \to T_n^{\text{max}}} \norm{\bvn(t)}^2_{L^2} = \infty.
\end{align}
This cannot be the case, as we established earlier that $\norm{\bvn(t)}^2_{L^2}$ is uniformly bounded.  This implies that we have $T_n^{\text{max}} = \infty$. Hence $\bvn$ is uniformly bounded (with respect to $n$) in $L^\infty(0,T;H)$ and $L^2(0,T;V^\alpha)$.

We wish to apply Aubin-Lions Compactness Theorem to extract classically convergent subsequences, however to do so we must first bound $\dv{\bvn}{t}$.
\begin{align}
&\quad
  \norm{\dv{\bvn}{t}}_{V^{-1}} 
  \\&\leq \notag
  \norm{{B(\bvn,\bvn)}}_{V^{-1}} + \nu \norm{A^\alpha \bvn }_{V^{-1}} + \mu \norm{I_h(\bu)}_{V^{-1}} + \mu\norm{I_h(\bvn)}_{V^{-1}} + \norm{\vect{f_n}}_{V^{-\alpha}}.
\end{align}

Now we will bound each term individually. To do so we bound the $L^2$ inner product of each term above with $\bw \in V^\alpha$.
\begin{align}
    \abs{\left\langle A^\alpha \bvn, \bw \right\rangle} & = \abs{\left\langle A^{\alpha/2}\bvn, A^{\alpha/2}\bw \right\rangle}\\
    &\leq C\norm{A^{\alpha/2}\bvn}_{L^2} \norm{A^{\alpha/2}\bw}_{L^2}\notag \\ \notag
    &= C\norm{\bvn}_{V^\alpha}\norm{\bw}_{V^\alpha}.
\end{align}
Similarly,
\begin{align}
    \abs{\left\langle I_h(\bvn), \bw \right\rangle} & \leq \norm{I_h(\bvn) }_{L^2}\norm{\bw}_{L^2}\\ \notag
    &\leq  \Calpha  \norm{\bvn}_{L^2}\norm{\bw}_{V^\alpha}.
\end{align}
Also,
\begin{align}
    \abs{\left\langle I_h(\bu), \bw \right\rangle} & \leq \norm{I_h(\bu )}\norm{\bw}
    \leq  \Calpha  \norm{\bu}\norm{\bw}_{V^\alpha}\notag \\ \notag
    & \leq  \Calpha  M\norm{\bw}_{V^\alpha}.
\end{align}
Finally,
\begin{align}
        \abs{b(\bvn,\bvn,\bw)} & = \abs{(b(\bvn, \bw, \bvn)}\\
        &\leq \norm{\bw}_{V}\norm{\bvn}_{L^4}^2\notag \\
        &\leq \norm{\bw}_{V}\norm{\bvn}_{V^\alpha}^{\frac{d}{2\alpha}}\norm{\bvn}^{2-\frac{d}{2\alpha}}\notag \\ \notag 
        &\leq  \Calpha  \norm{\bw}_{V^\alpha}\norm{\bvn}_{V^\alpha}^{\frac{d}{2\alpha}}\norm{\bvn}^{2-\frac{d}{2\alpha}}.
\end{align}
Now we know by the previous estimates that $\set{\bvn}_{n\in\nN}$ is bounded in $L^\infty(0,T; H)\cap L^2(0,T;V^\alpha)$, and thus $\set{\dv{\bvn}{t}}_{n\in\nN}$ is bounded in $L^{\frac{4\alpha}{d}}(0,T; V^{-\alpha})$.

Now we can apply the Banach-Aloaglu Theorem and Aubin-Lions Compactness Lemma to extract subsequences (relabeling when necessary) such that:
\begin{align}
    \bvn \weakstarto &\bv \text{ weak-$*$ in } L^\infty(0,T;H), \label{weak*}\\    
    \bvn \weaklyto &\bv \text{ weakly in } L^2(0,T;V^\alpha),\label{weak}\\
    \bvn \to &\bv \text{ strongly in } L^2(0,T; H). \label{strong}
\end{align}

Now it remains to show that each term of the Galerkin system converges in $L^2(0,T;V^{-\alpha})$ to the equivalent terms in \cref{eq: NSE AOT functional}. Let $\bw \in C^1(0,T;V^\alpha)$ be an arbitrary test function such that $\bw(T) = 0$. Now we take the inner product of \cref{galerkin system} with $\bw$, integrate from $0$ to $T$, and integrate by parts to obtain the following equations for each $n\in\nN$:
\begin{align}
    -(\bvn(0), \bw(0)) - \int_0^T (\bvn, \bw_t)dt 
    + \int_0^T(P_n B(\bvn,\bvn),\bw)dt + \nu \!\int_0^T (A^{\alpha/2} \bvn, A^{\alpha/2}\bw)dt \label{weak eq Galerkin}\\
    = \mu \int_0^T ( P_n I_h(\bu), \bw)dt - \mu \int_0^T(P_n I_h(\bvn), \bw)dt + \int_0^T (\vect{f}_n, \bw)dt .\nonumber
\end{align}
Since $\bw \in C^1(0,T;V^\alpha)$, using \cref{weak}, it is straight-forward to show the following results for the linear terms:
\begin{align}
    \nu \int_0^T (A^{\alpha/2} \bvn, A^{\alpha/2}\bw)dt &\to \nu\int_0^T (A^{\alpha/2} \bv, A^{\alpha/2}\bw)dt, \\
    \int_0^T (\bvn, \bw_t)dt &\to \int_0^T (\bv, \bw_t)dt,\notag \\ \notag
   \mu \int_0^T(P_n I_h(\bvn), \bw)dt &\to \mu \int_0^T(P_n I_h(\bv),\bw)dt.
\end{align}

We now need to check the convergence of the trilinear form, $I_n \to 0$, where: 
\begin{align}
I_n:= \int_0^T (B(\bvn,\bvn),P_n\bw)dt - \int_0^T (B(\bv,\bv),\bw)dt. 
\end{align}
Note that here we have rewritten the trilinear term, using the fact that projection operators are self-adjoint.

To do this we rewrite the trilinear form, denoting $\bwn:=P_n\bw$, as:
\begin{align}
    I^n &= \underbrace{\int_0^T (B(\bv - \bvn,\bw),\bv)dt}_{I_1^n}\\
    &\quad+ \underbrace{\int_0^T (B(\bvn,\bw),\bv - \bvn)dt}_{I_2^n}\\
    &\quad+ \underbrace{\int_0^T (B(\bvn,\bw - \bwn),\bv)dt}_{I_3^n}.
\end{align}
We note that $\abs{I_1^n} = \abs{\int_0^T(B(\bv - \bvn,\bw),\bv)dt} \to 0$ as $n \to \infty$ follows from the weak-$*$ convergence of $\bvn$ to $\bv$ in $L^\infty(0,T;H)$. To justify this rigorously, we note that since $\bvn \weakstarto \bv$ in $L^\infty(0,T;H)$ we need only to show that each component $v_i\pd{w_k}{x_j}  \in L^1(0,T;H)$ for $i,j,k\in\set{1,\ldots,d}$. 
To this end, we estimate, for each $i,j,k\in\set{1,\ldots,d}$,
\begin{align}
    \norm{v_i\pd{w_k}{x_j}}_{L^1(0,T;H)} & = \int_0^T \norm{v_i\pd{w_k}{x_j}}_{L^2} dt\\
    & \leq \int_0^T \Calpha\norm{v_i}_{L^{\frac{2d}{d-2\alpha}}}\norm{\pd{w_k}{x_j}}_{L^{\frac{d}{\alpha}}}dt\notag \\
    &\leq \int_0^T \Calpha \norm{v_i}_{V^\alpha}\norm{\pd{w_k}{x_j}}_{V^{\frac{d}{2} - \alpha}}dt\notag \\
    &\leq \int_0^T \Calpha\norm{\bv}_{V^\alpha}\norm{\bw}_{V^{\frac{d}{2} - \alpha + 1}}dt\notag \\
    &\leq \int_0^T \Calpha
    \norm{\bv}_{V^\alpha}\norm{\bw}_{V^\alpha} dt\notag \\
    &\leq \frac{ \Calpha }{2}\int_0^T\norm{\bv}_{V^\alpha}^{2}dt + \frac{ \Calpha }{2}\int_0^T\norm{\bw}_{V^\alpha}^{2}dt\notag \\\notag
    &\leq  \frac{ \Calpha }{2} \norm{\bv}_{L^2(0,T;V^\alpha)}^2 + T\frac{ \Calpha }{2}\norm{\bw}_{L^\infty(0,T;V^\alpha)}.
\end{align}
Note here that to obtain $\norm{\bw}_{V^{\frac{d}{2} +1 - \alpha }} \leq \norm{\bw}_{V^\alpha}$, we used $\frac{d}{2} - \alpha + 1 \leq \alpha$, which simplifies to $\frac{d}{4} + \frac{1}{2} \leq \alpha$.
We additionally note that this estimate applies only for $\alpha < d$ due to the H\"older conjugates chosen for the second line of estimates. For larger $\alpha$ values we instead use the following estimate, which takes advantage of Agmon's inequality:
\begin{align}\label{est: large alpha L2}\norm{v_i\pd{w_k}{x_j}}_{L^2} \leq C\norm{\bv}_{L^\infty}\norm{\bw}_{L^2} \leq C \norm{\bv}_{V^\alpha}\norm{\bw}_{V^\alpha}.\end{align}
From here we can proceed similarly to achieve the desired bounds.
Thus we have $v_i\pd{w_k}{x_j} \in L^1(0,T;V^\alpha)$ and thus $\abs{I_1^n} \to 0$ for each $i,j,k\in\set{1,\ldots,d}$ as expected.

Now to show $\abs{I_2^n}\to 0$ we will use the fact that $\bvn \to \bv $ strongly in $L^2(0,T;H)$.
\begin{align}
    \abs{I_2^n} &= \abs{\int_0^T (B(\bvn,\bw),\bv - \bvn)dt}\\
     &\leq \int_0^T \abs{(B(\bvn,\bw),\bv - \bvn)}dt\notag \\
     &\leq \int_0^T C\norm{\bvn}_{L^{\frac{2d}{d-2\alpha}}}\norm{\nabla \bw}_{L^{\frac{d}{\alpha}}}\norm{\bv - \bvn}_{L^2}dt\notag \\ 
     &\leq \int_0^T C\norm{\bvn}_{V^\alpha}\norm{\nabla \bw}_{V^{\frac{d}{2} - \alpha}}\norm{\bv - \bvn}_{L^2}dt\notag \\    
     &\leq \int_0^T C\norm{\bvn}_{V^\alpha}\norm{\bw}_{V^{\frac{d}{2} - \alpha + 1}}\norm{\bv - \bvn}_{L^2}dt\notag \\
     &\leq \int_0^T  \Calpha  \norm{\bvn}_{V^\alpha}\norm{\bw}_{V^{\alpha}}\norm{\bv - \bvn}_{L^2}dt\notag \\ \notag
     &\leq  \Calpha \norm{\bvn}_{L^2(0,T;V^\alpha)}\norm{\bw}_{L^\infty(0,T;V^\alpha)}\norm{\bv-\bvn}_{L^2(0,T;H)}.
\end{align}

We note that once again, this estimate is only valid for $\alpha \leq d$. Using \cref{est: large alpha L2} once again, we can extend this estimate to higher values of $\alpha$.
Thus we have that $\abs{I_2^n} \to 0$ as $\bvn$ is uniformly bounded in $L^2(0,T;V^\alpha)$, $\bw \in C(0,T;V^\alpha)$, and the strong convergence of $\bvn\to \bv$ in $L^2(0,T;H)$.

Finally, to show $\abs{I_3^n} \to 0$ we use the fact that $P_n \bw \to \bw$ in $L^\infty(0,T;V^\alpha)$ as $\bw \in C(0,T;V^\alpha)$.
\begin{align}
     \abs{I_3^n} &= \abs{\int_0^T (B(\bvn,\bw - P_n \bw),\bv)dt}\\
     &\leq \int_0^T \abs{(B(\bvn,\bw - P_n \bw),\bv)}dt\notag \\
     &\leq \int_0^T C\norm{\bvn}_{L^{\frac{2d}{d-2\alpha}}}\norm{\nabla (\bw - P_n \bw)}_{L^{\frac{d}{\alpha}}}\norm{\bv }_{L^2}dt\notag \\ 
     &\leq \int_0^T C\norm{\bvn}_{V^\alpha}\norm{\nabla (\bw - P_n \bw)}_{V^{\frac{d}{2} - \alpha}}\norm{\bv}_{L^2}dt\notag \\    
     &\leq \int_0^T C\norm{\bvn}_{V^\alpha}\norm{\bw - P_n \bw}_{V^{\frac{d}{2} - \alpha + 1}}\norm{\bv}_{L^2}dt\notag \\
     &\leq \int_0^T  \Calpha  \norm{\bvn}_{V^\alpha}\norm{\bw - P_n \bw}_{V^{\alpha}}\norm{\bv}_{L^2}dt\notag \\ \notag
     &\leq  \Calpha  \norm{\bvn}_{L^2(0,T;V^\alpha)}\norm{\bw - P_n\bw}_{L^\infty(0,T;V^\alpha)}\norm{\bv}_{L^2(0,T;H)}.
\end{align}






As with all of the previous estimates, this estimate remains valid for $\alpha \leq d$, but can be extended to higher values by use of \cref{est: large alpha L2}.
Thus we have that $\abs{I^n} \to 0$ as $n \to \infty$.




Thus, we have that $\bv$ solves the following system, for all $\bw \in C^1(0,T; L^2)$ with $\bw(T) = 0$: 
\begin{align}
    -(\bv(0), \bw(0)) - \int_0^T (\bv, \bw_t)dt + \int_0^T(B(\bv,\bv),\bw)dt + \nu \int_0^T (A^{\alpha/2} \bvn, A^{\alpha/2}\bw)dt\label{weak eq}\\
    = \mu \int_0^T ( I_h(\bu), \bw)dt - \mu \int_0^T(I_h(\bv), \bw)dt.\nonumber
\end{align}

Thus, we have that $\bv$ is a solution to \cref{eq: NSE AOT functional} in the sense of $L^2(0,T;V^{-\alpha})$. It follows from comparison with \cref{weak eq Galerkin} that $\bv(0) = \bv_0$. By subtracting \cref{weak eq Galerkin} and \cref{weak eq}, we obtain, after passing to the limits,
\begin{align}
    (\bv_0 - \bv(0), \vect{\psi}) = 0.
\end{align}
Here $\bw(0) = \vect{\psi} \in V^\alpha$. We note that since $\bw(0)\in C(0,T;V^\alpha)$ was arbitrary, so too is $\vect{\psi}$. Thus, as this holds for any $\vect{\psi} \in V^\alpha$ we have that $\bv_0 = \bv(0)$, and so the initial condition is indeed satisfied.

Finally, we want to show that $\bv \in C_w(0,T;H)$, that is $\bv$ is weakly continuous. To do so, we first note that $\bv \in L^\infty(0,T;H)$, which we obtained previously. 
It remains to show that $(\bv(t), \bphi)$ is a continuous function for any $\bphi \in H$. 
We note, that by \cref{strong,weak,weak*}, that there exists a Lebesgue measureable set $E$ such that for any $t\in [0,T]\setminus E$,
\begin{align}    
    \bvn(t) \weaklyto \bv(t) \text{ weakly in } V^\alpha \text{ and }    \bvn(t) \to \bv(t) \text{ strongly in } H.
\end{align}
 Now let $\bphi \in V^\alpha$ be given, and we take the inner product of \cref{weak eq Galerkin} with $\bphi$ and integrate from $t_0$ to $t$ for $t_0 \in [0,T]\setminus E$. This yields,
\begin{align}
    (\bvn(t),\bphi) -(\bvn(t_0), \bphi)  + \int_{t_0}^t(P_n B(\bvn,\bvn),\bphi)ds + \nu \int_{t_0}^t (A^{\alpha/2} \bvn, A^{\alpha/2}\bphi)ds\\
    = \mu \int_{t_0}^t ( I_h(\bu), \bphi)ds - \mu \int_{t_0}^t(I_h(\bvn), \bphi)ds.\nonumber
\end{align}

Now, by similar arguments to the ones above, we obtain the following equation for $\bv$.
\begin{align}
    (\bv(t),\bphi) -(\bv(t_0), \bphi)  + \int_{t_0}^t(B(\bv,\bv),\bphi)ds + \nu \int_{t_0}^t (A^{\alpha/2} \bv, A^{\alpha/2}\bphi)ds\\ \notag
    = \mu \int_{t_0}^t ( I_h(\bu), \bphi)ds - \mu \int_{t_0}^t(I_h(\bv), \bphi)ds.\nonumber
\end{align}


Now, we can send $t\to t_0$ using values $t\in [0,T]\setminus E$ and we have that 
\begin{align}
\lim_{t\to t_0} (\bv(t),\bphi) = (\bv(t_0),\bphi).
\end{align}
Since this holds for all $\bphi \in V^\alpha$, it must also hold for all $\bphi \in H$ as $V^\alpha$ is dense in $H$. Therefore we have that $\bv\in C_w(0,T;H)$.
\end{proof}

We note that thus far we have only established global existence of weak solutions. We have not established uniqueness or continuous dependence on initial data. Instead of attempting to prove these statements for weak solutions, we will move to establish existence of more regular solutions. 


\begin{theorem}\label{thm - classical Existence}  Assume the hypotheses of \Cref{thm: weak}.
If, in addition, it holds that $\bv_0\in V^{2\alpha}$, then the weak solution $\bv$ of \cref{eq: NSE AOT functional} is in fact a classical solution. That is, $\bv \in L^\infty(0,T;V^{2\alpha})\cap L^2(0,T;V^{3\alpha})$.
\end{theorem}

\begin{proof}
Here, we only show formal estimates, but these can be made rigorous by using Galerkin methods as in the proof of \Cref{thm: weak}.
We first take the inner product of the equation \cref{eq: NSE AOT functional} with $A \bv$. After (formally) integrating by parts we obtain the following equation:
\begin{align}
\frac{1}{2}\dv{~}{t}\norm{\bv}_{V}^2 + \nu\norm{A^{\alpha/2 + 1}\bv}^2_{L^2}&= \sum_{i,j,k=1}^{d}\int_\Omega \pdv{v_i}{x_k}\pdv{v_j}{x_i} \pdv{v_j}{x_k} \, d\bx + \left\langle\vect{f}, \bv\right\rangle 
\\ \notag&\qquad \notag
+ \mu \left\langle I_h(\bu), A \bv \right\rangle - \mu \left\langle I_h(\bv), A \bv \right\rangle.
\end{align}



We now estimate all of the terms on the right-hand side as follows:

\begin{align}
    - \mu \left\langle I_h(\bv), A \bv \right\rangle &=  \mu \left\langle \bv - I_h(\bv), A \bv \right\rangle + \mu \left\langle \bv, A \bv \right\rangle\\ \notag
    &\leq \mu \norm{\bv - I_h(\bv)}_{L^2} \norm{A \bv}_{L^2}- \mu \norm{\bv}^2_{V}\\ \notag
    &\leq \frac{\mu^2 \Calpha }{\nu} \norm{\bv - I_h(\bv)}^2_{L^2} + \frac{\nu}{4 \Calpha }\norm{A \bv}_{L^2}^2 - \mu \norm{\bv}^2_{V}\\ \notag
    &\leq \frac{\mu^2 c_{0,\alpha} h^2 \Calpha  }{\nu}\norm{\bv}_{V}^2 + \frac{\nu}{4 \Calpha }\norm{ \bv}^2_{V^2} - \mu \norm{\bv}_{V}^2\\ \notag
    &\leq \frac{\nu}{4}\norm{\bv}_{V^{\alpha+1}}^2.
\end{align}
Also,
\begin{align}
\mu \abs{\left\langle I_h(\bu),A \bv\right \rangle} &\leq \mu \norm{I_h(\bu)}_{L^2} \norm{A\bv}_{L^2}\\ \notag
&\leq \frac{2\mu  \Calpha  }{\nu}\norm{I_h(\bu)}^2_{L^2} + \frac{\nu}{8 \Calpha }\norm{A\bu}^2_{L^2}\\ \notag
&\leq  \Calpha  M + \frac{\nu}{16} \norm{\bv}^2_{V^{\alpha + 1}}.
\end{align}
Moreover,
\begin{align}
    \abs{\left\langle\vect{f}, A  \bv \right\rangle} &\leq   \Cnu \norm{\vect{f}}_{L^2}\norm{\bv}_{V^2}\\ \notag
    &\leq   C_{\alpha,\nu}  \norm{\vect{f}}_{L^2}^2 + \frac{\nu}{16  \Calpha }\norm{\bv}_{V^2}^2\\ \notag
    &\leq   C_{\alpha,\nu}  \norm{\vect{f}}_{L^2}^2 + \frac{\nu}{16}\norm{\bv}_{V^{\alpha + 1}}^2.
\end{align}
Finally,
\begin{align}\label{vvv_bnd}
     \abs{\sum_{i,j,k=1}^{d}\int_\Omega \pdv{v_i}{x_k}\pdv{v_j}{x_i} \pdv{v_j}{x_k}\,  d\bx} &\leq C\norm{\nabla \bv}_{L^3}^3\\ \notag
     &\leq C\norm{\nabla \bv}^{3a_1}_{L^2}\norm{A^{\frac{\alpha}{2}}\bv}^{3a_2}_{L^2}\norm{A^{\frac{\alpha + 1}{2}}\bv}^{3a_3}_{L^2}\\ \notag
     &\leq \frac{\nu}{8}\norm{\bv}^{2}_{V^{\alpha+1}} + C\norm{ \bv}^{2}_{V}\norm{\bv}^{2\frac{a_2}{a_1}}_{V^\alpha}.
\end{align}
Here $(a_1,a_2,a_3) = (1 - \gamma, \frac{1}{3}, \gamma - \frac{1}{3})$, where $\gamma = \frac{1}{3\alpha}(1 + \frac{d}{2})$. Note that here $M>0$ is a constant such that $\norm{I_h(\bu)}_{L^\infty(0,T;H)}^2\leq M$. Such an $M$ exists since $\bu \in L^\infty(0,T;H)$ and thus $I_h(\bu)\in L^\infty(0,T;H)$ as well.
Now we note that bound \cref{vvv_bnd} is only applicable for $\alpha < \frac{d}{2} + 1$.



Hence, we obtain the following estimate:
\begin{align}
    \dv{~}{t}\norm{\bv}_{V}^2 + \nu\norm{\bv}^2_{V^{\alpha+1}} \leq  \Calpha \norm{\bv}^{2\frac{a_2}{a_1}}_{V^\alpha} \norm{\bv}_{V}^2 + CM.
\end{align}

We recall from the previous section, that $\bv \in L^\infty(0,T;H)\cap L^2(0,T;V^\alpha)$. Using this and the fact that $\frac{a_2}{a_1} \leq 1$ for $\alpha \geq \frac{d}{4} + \frac{1}{2}$, we obtain the following via H\"older's inequality,
\begin{align}
    \int_0^t \norm{\bv}_{V^\alpha}^{2\frac{a_2}{a_1}}ds &\leq \left(\int_0^T \norm{\bv}_{V^\alpha}^{2 \frac{a_2}{a_1} (\frac{a_1}{a_2})}dt\right)^{\frac{a_2}{a_1}}\left( \int_0^T 1^{\frac{1}{1-\frac{a_2}{a_1}}}\right)^{1-\frac{a_2}{a_1}}\\\notag
    &= T^{1-\frac{a_2}{a_1}}\left(\int_0^T \norm{\bv}_{V^\alpha}^2\right)^{\frac{a_1}{a_2}} dt
    < \infty.
\end{align}
Thus it follows straightforwardly from an application of Gr\"onwall's inequality that $\bv \in L^\infty(0,T;V)\cap L^2(0,T;V^{\alpha + 1})$.


We note that we can obtain a similar bound on the nonlinear term for $\alpha >\frac{d}{2} + 1$ as follows:
\begin{align}
    \abs{b(\bv,\bv,A \bv)}& \leq C\norm{\bv}_{L^\infty}\norm{\bv}_{V} \norm{\bv}_{V^2}\\ \notag
    &\leq C \norm{\bv}_{V^{\frac{d}{2} - 1}}^{1/2}\norm{\bv}_{V^{\frac{d}{2}+1}}^{1/2}\norm{\bv}_{V}\norm{\bv}_{V^2}\\ \notag
    &\leq  \Calpha  \norm{\bv}_{V^\alpha}\norm{\bv}_{V}\norm{\bv}_{V^{\alpha+1}}\\ \notag
   &\leq   C_{\alpha,\nu} \norm{\bv}_{V^\alpha}^2\norm{\bv}_{V}^2 + \frac{\nu}{2}\norm{\bv}_{V^{\alpha + 1}}^2.
\end{align}

This yields a closed estimate as $\int_0^t \norm{\bv}_{V^\alpha}^{2} ds < \infty$ for $\alpha \geq \frac{d}{2} + 1$. Thus we have obtained  $\bv \in L^\infty(0,T;V)$ for any choice of $\alpha \geq \frac{d}{4} + \frac{1}{2}$.

Before we move onto the inductive step, we additionally require $\bv \in L^\infty(0,T;V^\alpha)$. To obtain such a bound we take the inner product of \cref{eq: NSE AOT functional} with $A^\alpha \bv$. This yields the following set of equations:
\begin{align}\label{eq: Halpha}
&\quad
    \frac{1}{2}\dv{~}{t}\norm{\bv}_{V^\alpha}^2 + \nu \norm{\bv}_{V^{2\alpha}}^2 + \left(B(\bv,\bv),A^\alpha \bv\right) 
    \\&\notag= (\vect{f},A^\alpha \bv) + \mu (I_h(\bu),A^\alpha \bv) - \mu (I_h(\bv),A^s\bv).
\end{align}

Using the Kato-Ponce inequality \cref{ineq:Kato-Ponce} we now estimate the nonlinear term as follows:
{\allowdisplaybreaks
\begin{align}
&\quad
    \abs{\left\langle B(\bv,\bv),A^\alpha \bv \right\rangle} 
    \\&= \notag
        \abs{
            \left\langle A^{\alpha/2} B(\bv,\bv),A^{\alpha/2} \bv \right\rangle
            }
= 
        \abs{
            \left\langle \Lambda^{\alpha} B(\bv,\bv),\Lambda^{\alpha} \bv \right\rangle
            }
    \\ \notag&\leq C 
            \norm{\Lambda^{\alpha} B(\bv,\bv)}_{L^{4/3}}
            \norm{\Lambda^{\alpha} \bv }_{L^4}
     \\ \notag&\leq C 
            \norm{\nabla \bv}_{L^2}\norm{\Lambda^{\alpha} \bv }^2_{L^4} 
        + C
            \norm{\bv}_{L^{8/3}}\norm{\Lambda^{\alpha}\nabla \bv}_{L^{8/3}}\norm{\Lambda^{\alpha} \bv }_{L^4}
    \\ \notag&= C 
            \norm{\nabla \bv}_{L^2}\norm{A^{\alpha/2} \bv }^2_{L^4} 
        + C
            \norm{\bv}_{L^{8/3}}\norm{A^{\alpha/2}\nabla \bv}_{L^{8/3}}\norm{A^{\alpha/2} \bv }_{L^4}
    \\ \notag&\leq C             
            \norm{\bv}_{V}\norm{\bv}_{V^\alpha}^{2\theta}
            \norm{\bv}_{V^{2\alpha}}^{2(1-\theta)} 
        + C
            \norm{\bv}_{V^{d/8}}\norm{\bv}_{V^{d/8 + \alpha}}
            \norm{\bv}_{V^\alpha}^\theta 
            \norm{\bv}_{V^{2\alpha}}^{1-\theta}
    \\ \notag&\leq 
    \frac{\nu}{24}
            \norm{\bv}_{V^{2\alpha}}^2 
        + C_{\alpha,\nu}
            \norm{\bv}_{V}^{\frac{1}{\theta}}
            \norm{\bv}_{V^{\alpha}}^2 
        + C
            \norm{\bv}_{V}\norm{\bv}_{V^{\alpha + 1}}
            \norm{\bv}_{V^\alpha}^\theta 
            \norm{\bv}_{V^{2\alpha}}^{1-\theta}
    \\ \notag&\leq 
        \frac{\nu}{24}
            \norm{\bv}_{V^{2\alpha}}^2 
        + C_{\alpha,\nu}
            \norm{\bv}_{V}^{\frac{1}{\theta}}
            \norm{\bv}_{V^{\alpha}}^2 
    +\frac{\nu}{24}
        \norm{\bv}_{V^{2\alpha}}^2 
    +  C_{\alpha,\nu}
        \norm{\bv}_{V}^2
        \norm{\bv}_{V^\alpha}^{2\theta} 
        \norm{\bv}_{V^{2\alpha}}^{2(1-\theta)}
    \\ \notag&\leq 
    \frac{2\nu}{24}
            \norm{\bv}_{V^{2\alpha}}^2 
        + C_{\alpha,\nu}
            \norm{\bv}_{V}^{\frac{1}{\theta}}
            \norm{\bv}_{V^{\alpha}}^2 
        +  C_{\alpha,\nu}
            \norm{\bv}_{V}^{\frac{2}{\theta}}
            \norm{\bv}_{V^\alpha}^{2} 
        +  \frac{\nu}{24}
            \norm{\bv}_{V^{2\alpha}}^{2}
    \\ \notag&\leq \frac{3\nu}{24}
        \norm{\bv}_{V^{2\alpha}}^2 
        + C_{\alpha,\nu}
            \left(
                    \norm{\bv}_{V}^{\frac{1}{\theta}}
                + 
                    \norm{\bv}_{V}^{\frac{2}{\theta}} 
            \right)
            \norm{\bv}_{V^{\alpha}}^2.
\end{align}
}

Note that above we are using the Sobolev inequalities $\norm{\bv}_{L^{8/3}} \leq \norm{\bv}_{V^{d/8}} \leq \norm{\bv}_{V}$, which will only remain valid for $d\leq 8$. 
\begin{remark}\label{rmk_dimension_restriction}
    This is the only part of the proof where the restriction $d\leq8$ appears.  It may be possible to push these results to higher dimensions, but doing so would require at least different estimates, so we do not pursue this here.
\end{remark}
We now perform similar estimates of the terms on the right-hand side of \cref{eq: Halpha}.
\begin{align}
    -\mu \langle I_h(\bv), A^s\bv \rangle &\leq \frac{\nu}{4}\norm{\bv}_{V^{2\alpha}},
\\
    \mu \langle I_h(\bu), A^s\bv \rangle &\leq C_\alpha M +  \frac{\nu}{16}\norm{\bv}_{V^{2\alpha}},
\\
    \langle\vect{f}, A^s\bv \rangle &\leq C_{\alpha,\nu}\norm{\vect{f}}^2_{L^2} + \frac{\nu}{16}\norm{\bv}_{V^{2\alpha}}.
\end{align}

Putting all of these estimates together we obtain the following inequality

\begin{align}
    \frac{1}{2}\dv{~}{t}\norm{\bv}_{V^\alpha}^2 +\frac{\nu}{2}\norm{\bv}_{V^{2\alpha}}^2 \leq C \left( \norm{\bv}_{V}^{\frac{1}{\theta}} + \norm{\bv}_{V}^{\frac{2}{\theta}} \right)\norm{\bv}_{V^{\alpha}}^2 + C_{\alpha,\nu}\norm{\vect{f}}^2_{L^2} + C_\alpha M.
\end{align}
From here we can apply Gr\"onwall's inequality to obtain a closed estimate with $\bv \in L^\infty(0,T;V^\alpha)$.

We can now proceed inductively to establish higher regularities.
To do so we will first take the action of $A^s \bv$ on \cref{eq: NSE AOT functional}  for $s$ to be determined later. This yields
\begin{align}
&\quad
    \frac{1}{2} \dv{~}{t}\norm{\bv}_{V^s}^2 + \nu \norm{\bv}_{V^{\alpha+s}}^2 + \langle B(\bv,\bv),A^s \bv \rangle 
    \\&= \notag
    \langle \vect{f},\bv \rangle + \mu \langle I_h(\bu),A^s\bv \rangle - \mu \langle I_h(\bv), A^s \bv\rangle. \label{inductive_estimate}
\end{align}

Now, we estimate each of the terms as follows:

\begin{align}
    \mu \langle I_h(\bu), A^s\bv\rangle &\leq \frac{\mu  \Calpha }{\nu}\norm{I_h(\bu)}_{L^2}^2 + \frac{\nu}{8  \Calpha }\norm{\bv}_{V^{2s}}^2\\ \notag
    &\leq \frac{\mu  \Calpha }{\nu}\norm{I_h(\bu)}_{L^2}^2 + \frac{\nu}{8}\norm{\bv}_{V^{s+\alpha}}^2,
\end{align}
also,
\begin{align}
    \abs{\langle\vect{f},A^s \bv\rangle}    & \leq  \norm{\vect{f}}_{L^2} \norm{\bv}_{V^{2s}}\\ \notag
                            & \leq    C_{\alpha,\nu}  \norm{\vect{f}}_{L^2}^2 + \frac{\nu}{8  \Calpha }\norm{\bv}_{V^{2s}}^2\\ \notag
                            & \leq    C_{\alpha,\nu}  \norm{\vect{f}}_{L^2}^2 + \frac{\nu}{8}\norm{\bv}_{V^{{s+\alpha}}}^2,
\end{align}

\noindent
and finally,
\begin{align}
    -\mu \langle I_h(\bv), A^s\bv \rangle &=  \mu \langle\bv - I_h(\bv), A^s\bv\rangle - \mu \langle\bv, A^s\bv\rangle,\\ \notag
    &= \mu \langle A^{s/2}(\bv - I_h(\bv)), A^{s/2}\bv\rangle - \mu\langle\bv, A^s\bv\rangle,\\ \notag
    &= \mu \langle A^{s/2}\bv - I_h(A^{s/2}\bv)), A^{s/2}\bv\rangle - \mu\langle\bv, A^s\bv\rangle,\\ \notag
    & \leq \frac{\mu c_{0,\alpha} h^2  \Calpha }{2}\norm{A^{s/2}\bv}_{V^\alpha}^2 + \frac{\mu}{2}\norm{\bv}^2_{V^{s}} - \mu \norm{\bv}^2_{V^{s}},\\ \notag
    &\leq \frac{\nu}{2} \norm{\bv}_{V^{s+\alpha}}^2.
\end{align}

We note that moving $A^{s/2} \bv$ onto the interpolant is necessary as otherwise one ends up with an estimate of the form:
\begin{align}
    -\mu\langle I_h(\bv), A^s\bv \rangle \leq \frac{\nu}{4}\norm{\bv}_{V^{2s}}^2.
\end{align}
Such an estimate will not be useful as we want to absorb the $V^{2s}$ term into the viscous term on the left side of \Cref{inductive_estimate}. 
This is because we fail to absorb this term for $s > \alpha$, which will be required to obtain the full result of \Cref{thm - classical Existence}.



It remains now to estimate the nonlinear term. This can be done as follows:
\begin{align}
    \abs{\langle B(\bv,\bv),A^s \bv \rangle} 
    &= 
        \abs{ 
            \langle \Lambda^{s}B(\bv,\bv), \Lambda^{s} \bv \rangle
            }
    \\ \notag
    &\leq C\norm{\Lambda^{s}B(\bv,\bv)}_{L^{4/3}}\norm{\Lambda^{s}\bv }_{L^4}\\ \notag
    &\leq C\left( \norm{\Lambda^{s}\bv}_{L^4}\norm{\nabla \bv}_{L^2} + \norm{\bv}_{L^4}\norm{\Lambda^{s}\nabla \bv}_{L^2} \right) \norm{\Lambda^{s}\bv }_{L^4}\\ \notag
    &= C \norm{A^{s/2}\bv}^2_{L^4}\norm{\nabla \bv}_{L^2} + C\norm{\bv}_{L^4}\norm{A^{s/2}\nabla \bv}_{L^2} \norm{A^{s/2}\bv }_{L^4}.
\end{align}

We note that above, we have utilized the Kato-Ponce inequality, which is applicable for Riesz derivatives, see, e.g., \cite{bae2015gevrey}. 
From here we now perform additional estimates on the right-hand side
\begin{align}
    C\norm{A^{s/2}\bv}^2_{L^4}\norm{\nabla \bv}_{L^2} 
    &\leq C\norm{A^{s/2}\bv}^{2\theta}_{L^2}\norm{A^{s/2}\bv}^{2(1-\theta)}_{V^\alpha}\norm{\bv}_{V}\\ \notag
    &\leq C\norm{\bv}^{2\theta}_{V^s}
    \norm{\bv}_{V^{s+\alpha}}^{2(1-\theta)} 
    \norm{\bv}_{V}\\ \notag
    &\leq   C_{\alpha,\nu} 
    \norm{\bv}_{V}^{\frac{1}{\theta}} 
    \norm{\bv}_{V^s}^2 
     + \frac{\nu}{24} \norm{\bv}_{V^{s+\alpha}}^2,
\end{align}
and
\begin{align}
&\quad
    C\norm{\bv}_{L^4}\norm{A^{s/2}\nabla \bv}_{L^2} \norm{A^{s/2}\bv }_{L^4} 
    \\&\leq\notag C\norm{\bv}_{V^\alpha}\norm{\bv}_{V^{s+1}}\norm{A^{s/2}\bv }_{L^2}^{\theta} \norm{A^{s/2}\bv }_{V^\alpha}^{1-\theta}\\ \notag
    &\leq C_{\alpha}  \norm{\bv}_{V^\alpha}\norm{\bv}_{V^{s+\alpha}}\norm{A^{s/2}\bv }_{L^2}^{\theta} \norm{A^{s/2}\bv }_{V^\alpha}^{1-\theta}\\ \notag
    &\leq \frac{\nu}{24} \norm{\bv}_{V^{s+\alpha}}^2 +   C_{\alpha,\nu}  \norm{\bv}_{V^\alpha}^2\norm{A^{s/2}\bv }_{L^2}^{2\theta} \norm{A^{s/2}\bv }_{V^\alpha}^{2(1-\theta)}\\ \notag
    &\leq \frac{2\nu}{24} \norm{\bv}_{V^{s+\alpha}}^2 +   C_{\alpha,\nu} \norm{\bv}_{V^\alpha}^{\frac{2}{\theta}}\norm{\bv }_{V^s}^{2}.
\end{align}

Combining these estimates we obtain:
\begin{align}
     \abs{\langle B(\bv,\bv),A^s \bv \rangle} &\leq \frac{\nu}{8}\norm{\bv}_{V^{s+\alpha}}^2 + C_{\alpha,\nu}(\norm{\bv}_{V}^{\frac{1}{\theta}} + \norm{\bv}_{V^\alpha}^{\frac{2}{\theta}})\norm{\bv}_{V^{s}}^2.
\end{align}
Here $\theta = 1 - \frac{d}{4\alpha}$. 
Thus we obtain the following bound.
\begin{align}
    \frac{1}{2}\dv{~}{t}\norm{\bv}^2_{V^{s}} 
    + \frac{\nu}{8}\norm{\bv}^2_{V^{s+\alpha}} 
    \leq K(\norm{\bv}_{V^{1}}^{\frac{1}{\theta}} + \norm{\bv}_{V^{\alpha}}^{\frac{2}{\theta}}) \norm{\bv}_{V^{s}}^{2} \\
    \quad + \frac{\mu  \Calpha }{\nu}\norm{I_h(\bu)}^2_{L^2} 
    +   C_{\alpha,\nu}  \norm{\vect{f}}^2_{L^2}.
\end{align}

Integrating in time we now obtain:    
\begin{align}
    \norm{\bv(t)}^2_{V^{s}} + \frac{\nu}{4}\norm{\bv}_{L^2(0,T;V^{s+\alpha})} 
    &\leq \norm{\bv_0}^2_{V^s} 
    +  C_{\alpha,\nu} M\norm{\bv}_{L^2(0,T;V^{s})}^{2}\\ \notag
    &\quad+ \frac{2\mu  \Calpha }{\nu}\norm{I_h(\bu)}^2_{L^2(0,T;H)} + 2  C_{\alpha,\nu}  \norm{\vect{f}}^2_{L^2(0,T;H)}.
\end{align}

Here $M>0$ is a constant such that $\norm{\bv}_{V}^{\frac{1}{\theta}} + \norm{\bv}_{V^\alpha}^{\frac{2}{\theta}} \leq M$. We know from before that $\bv \in L^\infty(0,T;V)\cap L^\infty(0,T;V^\alpha)$, so we know that there exists such an $M$.
Thus we obtain that $\bv \in L^\infty(0,T;V^s)\cap L^2(0,T;V^{s+\alpha})$ so long as $\bv \in L^\infty(0,T;V)\cap L^2(0,T;V^s)$ with $\bv_0 \in V^s$. 
We can thus iterate $s$ values to obtain that 
$\bv \in L^\infty(0,T;V^{2\alpha})\cap L^2(0,T;V^{3\alpha})$ for initial data $\bv_0 \in V^{2\alpha}$.

We note that the above formal proof utilizes the following frequently, for various values of $s\geq 1$:
\begin{align}
    \left(\bv_t, A^s\bv\right) = \frac{1}{2}\dv{~}{t}\norm{\bv}_{V^s}^2. \label{LM-Lemma}
\end{align}
For this statement to be true we require $\bv_t \in L^2(0,T;V^{-s})$.
To show this we can return to the equation \cref{eq: NSE AOT functional} and show that each term is bounded in $L^2(0,T;V^{-s})$, under the additional assumption that $\bv \in L^\infty(0,T;V^{2\alpha})$. 
For the sake of simplicity, we bound only the nonlinear term, as it is straightforward to achieve such estimates for each of the linear terms, as we do above in the proof of \Cref{thm: weak}.

Let $\bw \in V^{s}$ be given such that $\norm{\bw}_{V^s} = 1$.  Then
\begin{align}
\abs{\left\langle B(\bv,\bv),\bw\right\rangle}
    &\leq C\norm{\bw}_{L^2}\norm{\bv}_{L^\infty}\norm{\nabla\bv}_{L^2}\\ \notag
    &\leq C_s \norm{\bw}_{V^s} \norm{\bv}_{V^{2\alpha}}\norm{\bv}_{V}
    \\ \notag
    &\leq C_s \norm{\bv}_{V^{2\alpha}}\norm{\bv}_{V}.
\end{align}
\noindent
Hence, since $\bv \in L^\infty(0,T;V^{2\alpha})\cap L^2(0,T;V)$, it follows that \newline$B(\bv,\bv) \in L^2(0,T;V^{-s})$, and hence from the equation, $\bv_t \in L^2(0,T;V^{-s})$.

We note the above logic may at first appear circular, as the argument showing that $\bv \in L^\infty(0,T;V^{2\alpha})$ requires \cref{LM-Lemma}. 
However, this assumption can be justified rigorously in the standard way using the Galerkin method.
Specifically, the above proof that $\bv\in L^\infty(0,T;V^{2\alpha})$ can be applied to a sequence of solutions $\bvn$ to the associated Galerkin system \cref{galerkin system} without modification. 
This will yield a sequence of solutions $\bvn$ uniformly bounded in $L^\infty(0,T;V^{2\alpha})$, which can be used to justify the above bound in $L^2(0,T;V^{-s})$ rigorously. 
\end{proof}

We now move to finish the proof of global regularity by establishing uniqueness and continuous dependence on data for classical solutions $\bv$ of \cref{eq: F AOT}.


\begin{theorem}
Assume the hypotheses of \Cref{thm: weak}, and that $\bv_0 \in V^{2\alpha}$ (so that \Cref{thm - classical Existence} applies).  Then \cref{eq: NSE AOT functional} admits a unique classical solution $\bv$ such that $\bv\in L^\infty(0,T;V^{2\alpha})\cap L^2(0,T;V^{3\alpha})$ and $\frac{d\bv}{dt} \in L^\frac{4\alpha}{d}(0,T;V^{-1})$.  Moreover, solutions depend continuously on the initial data in the $L^\infty(0,T;H)$ and $L^2(0,T;V^\alpha)$ norms.
\end{theorem}
\begin{proof}
Existence of classical solutions to \cref{eq: NSE AOT functional} is already given by \Cref{thm - classical Existence}.
Let $\bv_1$ and $\bv_2$ be two such solutions with initial data $\bv_1(0) = \bv_2(0)$, respectively. Let $\bw := \bv_1 - \bv_2$ and note that $\bw$ satisfies
\begin{align}
    \bw_t + \nu A^\alpha \bw + B(\bv_1,\bw) + B(\bw,\bv_2) - B(\bw,\bw) = -\mu I_h(\bw).\label{uniqueness}
\end{align}

We now take the action of \cref{uniqueness} on $\bw$ to obtain the following equations
\begin{align}
&\quad
    \frac{1}{2}\dv{~}{t}\norm{\bw}^2_{L^2} + \nu\norm{\bw}_{V^\alpha}^2 + b(\bv_1,\bw, \bw) + b(\bw, \bv_2, \bw) - b(\bw,\bw, \bw) 
    \\&=\notag
    -\mu \langle I_h(\bw),\bw\rangle,
\end{align}
where we used the Lions-Magenes Lemma to pull out the time derivative.
From \cref{symm2} we have that $b(\bv_1,\bw,\bw) = b(\bw,\bw,\bw) = 0$, thus only one nonlinear term remains.  
Let $K = 2\max(\norm{\bv_1}_{L^\infty(0,T;V^{2\alpha})}, \norm{\bv_2}_{L^\infty(0,T;V^{2\alpha})})$.
Note that $\norm{\bw}_{V^{2\alpha}} \leq K$ as $\bw = \bv_1 - \bv_2$. 
Estimating the nonlinear and interpolant terms, we obtain
\begin{align}
\abs{b(\bw,\bv_2,\bw)} & = \abs{b(\bw,\bw,\bv_2)}\\ \notag
&\leq C\norm{\bv_2}_{L^\infty}\norm{\grad \bw}_{L^2}\norm{\bw}_{L^2}\\ \notag
&\leq C\norm{\bv_2}_{V^{2\alpha}}\norm{\bw}_{V}\norm{\bw}_{L^2}
\leq  \Calpha\norm{\bv_2}_{V^{2\alpha}}\norm{\bw}_{V^\alpha}\norm{\bw}_{L^2}\\ \notag
&\leq \frac{\nu}{4}\norm{\bw}^2_{V^\alpha} +  C_{\alpha,\nu}\norm{\bv_2}^2_{V^{2\alpha}}\norm{\bw}^2_{L^2}\\ \notag
&\leq \frac{\nu}{4}\norm{\bw}^2_{V^\alpha} +  C_{\alpha,\nu} K^2\norm{\bw}^2_{L^2},
\end{align}
and 
\begin{align}
    -\mu (I_h(\bw),\bw) &= \mu(\bw - I_h(\bw), \bw) - \mu(\bw,\bw)\\ \notag
    &\leq \mu\norm{\bw - I_h(\bw)}_{L^2}\norm{\bw}_{L^2}- \mu\norm{\bw}_{L^2}^2\\ \notag
    &\leq \frac{\mu}{2}\norm{\bw - I_h(\bw)}_{L^2}^2 +  \frac{\mu}{2}\norm{\bw}_{L^2}^2- \mu\norm{\bw}_{L^2}^2\\ \notag
    &\leq \frac{\mu c_{0,\alpha} h^2  \Calpha }{2}\norm{\bw}_{V^\alpha}^2\\ \notag
    &\leq \frac{\nu}{2}\norm{\bw}_{V^\alpha}^2.
\end{align}


Gathering these bounds and rearraging, we obtain,
\begin{align}\label{w_bound}
\dv{~}{t}\norm{\bw}_{L^2}^2+2\nu\norm{\bw}_{V^\alpha}^2\leq 2 C_{\alpha,\nu} K^2 \norm{\bw}_{L^2}^2.
\end{align}

Dropping the viscous term and using Gr\"onwall's inequality yields $\norm{\bw(t)}_{V^s}^2 \leq \norm{\bw(0)}_{L^2}^2e^{2 C_{\alpha,\nu} K^2t}$. Using this in \cref{w_bound} and integrating in time yields
\begin{align}\label{w_bound_int}
\norm{\bw(t)}_{L^2}^2+2\nu\int_0^t\norm{\bw(s)}_{V^\alpha}^2\,ds\leq  \norm{\bw(0)}_{L^2}^2(e^{2 C_{\alpha,\nu} K^2t}+1).
\end{align}

If $\bv_1(0)\equiv\bv_2(0)$, then $\bw(0) \equiv 0$, and it follows that that $\bv_1 \equiv \bv_2$.  Hence solutions are unique.
Moreover, \cref{w_bound_int} implies continuous dependence on the initial data in the desired norms.
\end{proof}

\section{Convergence to Reference Solution}
\label{convergence}

\noindent
Now that we have established well-posedness of the solutions to \cref{eq: NSE AOT functional} we work to show that the solution $\bv$ converges to the reference solution $\bu$. To do this we will set $\bw = \bu - \bv$. We note that $\bw$ solves the following equation:
\begin{align}\label{eq: w}
    \bw_t + B(\bu, \bu) - B(\bv, \bv) + \nu A^\alpha \bw = \mu I_h(\bw).
\end{align}

We can then rewrite the nonlinear terms as 
\begin{align}
B(\bu, \bu) - B(\bv, \bv) = B(\bu, \bw) + B(\bw, \bu) - B(\bw, \bw).
\end{align}

Thus we obtain the following equation for $\bw$:

\begin{align}
    \bw_t + B(\bu, \bw) + B(\bw, \bu) - B(\bw, \bw)+ \nu A^\alpha \bw = \mu I_h(\bw).\label{Difference}
\end{align}

This equation will be our starting place for our convergence results, detailed below.


\begin{theorem}\label{thm-L2}
Let $\bu$ and $\bv$ be classical solutions to \cref{eq: NSE functional,eq: NSE AOT functional}, respectively, with initial data $\bu_0, \bv_0 \in V^{2\alpha}$ and $\bv$ defined on $[t_0,\infty)$.
Suppose $I_h$ is a linear interpolant satisfying \cref{eq: interp}, $\bbf\in H$ is a time-independent forcing, and $\mu,h$ satisfy the following conditions:
\begin{align}
    \mu > 2 C_{\alpha,\nu}\norm{\bu}_{V}^{\frac{1}{\theta}},
\end{align}
(where $\theta = 1- \frac{d}{4\alpha}$)
and
\begin{align}
            \mu h^2 c_{0,\alpha} \leq \nu, 
\end{align}
Then $\norm{\bu-\bv}_{L^2}$ decays to $0$ exponentially fast in time.



\end{theorem}

\begin{proof}
We begin by taking the inner product of \cref{eq: w} with $\bw$. This yields,
\begin{align}
    \frac{1}{2}\dv{~}{t}\norm{\bw}^2_{L^2} + b(\bu, \bw, \bw) + b(\bw, \bu, \bw) - b(\bw, \bw, \bw) + \norm{A^{\frac{\alpha}{2}}\bw}^2_{L^2} = -\mu \left\langle I_h(\bw), \bw \right\rangle.
\end{align}

Now, by symmetry we can eliminate two of the trilinear terms giving the following equation
\begin{align}
        \frac{1}{2}\dv{~}{t}\norm{\bw}^2_{L^2} + b(\bw, \bu, \bw) + \nu \norm{A^{\frac{\alpha}{2}}\bw}^2_{L^2} = -\mu \left\langle I_h(\bw), \bw \right\rangle.
\end{align}

We now move to estimate each of the terms, yielding the following estimates for the nonlinear and interpolant terms, respectively:
\begin{align}
   \abs{ (B(\bw,\bu),\bw) }
&\leq C\norm{\bw}^2_{L^4}\norm{\nabla \bu}_{L^2}\\ \notag 
&\leq C\norm{\bw}^{2\theta}_{L^2}\norm{\bw}_{V^\alpha}^{2(1-\theta)}\norm{\bu}_{V}\\ \notag
&\leq  C_{\alpha,\nu}\norm{\bu}_{V}^{\frac{1}{\theta}} \norm{\bw}^2_{L^2} + \frac{\nu}{2} \norm{\bw}^2_{V^\alpha},
\end{align}

and

\begin{align}
-\mu \left\langle I_h(\bw), \bw \right\rangle
&= \mu \left\langle \bw - I_h(\bw), \bw \right \rangle - \mu\left\langle \bw, \bw \right\rangle\\ \notag
&\leq \mu \norm{\bw - I_h(\bw)}_{L^2}\norm{\bw}_{L^2} - \mu \norm{\bw}^2_{L^2}\\ \notag
&\leq \frac{\mu}{2} \norm{\bw - I_h(\bw)}^2_{L^2} + \frac{\mu}{2}\norm{\bw}^2_{L^2} - \mu \norm{\bw}^2_{L^2}\\ \notag
&\leq \frac{\mu c_{0,\alpha} h^2}{2} \norm{\bw }_{V^\alpha}^2 + \frac{\mu}{2}\norm{\bw}^2_{L^2} - \mu \norm{\bw}^2_{L^2}\\ \notag
&\leq \frac{\nu}{2} \norm{\bw }_{V^\alpha}^2 - \frac{\mu}{2} \norm{\bw}_{L^2}^2.
\end{align}

Inserting these estimates and simplifying the resulting equations we now obtain
    \begin{align}
       \dv{~}{t}\norm{\bw}_{L^2}^2 + \left(\mu - 2 C_{\alpha,\nu}\norm{\bu}_{V}^{\frac{1}{\theta}}\right)\norm{\bw}^2_{L^2} \leq 0. 
    \end{align}
 Now, by Gr\"onwall inequality, we have that $\norm{\bw}^2_{L^2} \to 0$ exponentially fast in $H$, for sufficiently large $\mu$ and so $\bv \to \bu$ in $H$ exponentially fast in time.
 \end{proof}

\begin{theorem}\label{thm-Halpha}
Let $\bu$ and $\bv$ be classical solutions to \cref{eq: NSE functional,eq: NSE AOT functional}, respectively, with initial data $\bu_0, \bv_0 \in V^{2\alpha}$ and $\bv$ defined on $[t_0,\infty)$.
Suppose $I_h$ is a linear interpolant satisfying \cref{eq: interp}, $\bbf\in V^\alpha$ is a time-independent forcing, and $\mu,h$ satisfy the following conditions:
\begin{align}
    \mu
    \geq C_{\alpha,\nu}\left(4\sigma_\alpha^{2/\theta} + \sigma_{1}^{1/\theta} + 2\sigma_{1}^{2/\theta} + \sigma_{\alpha+1}^{\footnotemark}\right), 
\end{align}\footnotetext{We note that $\sigma_{\alpha+1}$ is not given in \cref{thm - unif bounds} as we only require only the bound $\norm{\bu}_{V^{\alpha+1}}\leq \sigma_{\alpha+1}$, which can be readily extracted from the regularity of $\bu$ for this initial data.} 
and
\begin{align}
            \mu h^2 c_{0,\alpha} \leq \nu,
\end{align}
where $\sigma_{s}$ are given such that $\norm{\bv}_{V^s}\leq \sigma_s$ and $\norm{\bu}_{V^s}\leq \sigma_s$, as in \cref{thm - unif bounds}.
Then $\norm{\bu-\bv}_{V^\alpha}$ decays to $0$ exponentially fast in time.
\end{theorem}

\begin{proof}
To prove this theorem, we will instead show convergence in $V^s$, for arbitrary $0 < s \leq \alpha$. Doing so will show convergence in $V^\alpha$ with $s = \alpha$, but can be easily adjusted to show convergence in $V$. Moreover, if one assumed more regular initial data, one can extend this proof straightforwardly to obtain convergence in more regular spaces.
As before, we take the inner product of \cref{eq: w} with $A^s \bw$.
This yields the following equation:
\begin{align}
&\quad
    \frac{1}{2}\dv{~}{t}\norm{\bw}_{V^s}^2 + \nu \norm{\bw}^2_{V^{s+\alpha}} + b(\bu,\bw, A^s \bw) + b(\bw,\bu,A^s\bw) - b(\bw,\bw,A^s\bw) 
    \\&\qquad= \notag
    -\mu \left< I_h(\bw),A^s \bw \right>.
\end{align}

Consider now the following estimates of the nonlinear terms.  For the first nonlinear term,
{\allowdisplaybreaks
\begin{align}
\abs{b(\bu,\bw,A^s\bw)} 
    &= 
        \abs{(\Lambda^{s}B(\bu,\bw),\Lambda^{s}\bw)}\\ \notag
    &\leq 
        \norm{\Lambda^{s}B(\bu,\bw)}_{L^{4/3}}
        \norm{\Lambda^{s} \bw}_{L^4}\\ \notag
    &\leq C
        \norm{\Lambda^{s} \bu}_{L^2}
        \norm{\nabla \bw}_{L^4}
        \norm{\Lambda^{s}\bw}_{L^4}
    \\ \notag
        &\quad+ C \notag
            \norm{\bu}_{L^4}
            \norm{\Lambda^{s}\nabla\bw}_{L^2}
            \norm{\Lambda^{s}\bw}_{L^4}
        \\ \notag
    &= C
        \norm{A^{s/2} \bu}_{L^2}
        \norm{\nabla \bw}_{L^4}
        \norm{A^{s/2}\bw}_{L^4}
    \\ \notag
        &\quad+ C \notag
            \norm{\bu}_{L^4}
            \norm{A^{s/2}\nabla\bw}_{L^2}
            \norm{A^{s/2}\bw}_{L^4}
        \\ \notag
        &\leq C
            \norm{\bu}_{V^s}
            \norm{\nabla \bw}_{L^4}
            \norm{A^{s/2}\bw}_{L^4}
        \\ \notag&\quad+ C \notag
            \norm{\bu}_{L^4}
            \norm{\bw}_{V^{s+1}}
            \norm{A^{s/2}\bw}_{L^4}
        \\ \notag
        &\leq C
            \norm{\bu}_{V^s}
            \norm{\bw}_{V^\alpha+1}
            \norm{\bw}^\theta_{V^{s}}
            \norm{\bw}^{1-\theta}_{V^{s+\alpha}}
      \\ \notag &\quad+ C \notag
            \norm{\bu}_{V^\alpha}
            \norm{\bw}_{V^{s+1}}
            \norm{\bw}_{V^{s}}^\theta
            \norm{\bw}_{V^{s+\alpha}}^{1-\theta}
        \\ \notag
    &\leq
        \frac{\nu}{20}
            \norm{\bw}_{V^{s+\alpha}}^2
        + C_\nu
            \norm{\bu}^2_{V^s}
            \norm{\bw}^{2\theta}_{V^s}
            \norm{\bw}^{2(1-\theta)}_{V^{s+\alpha}}
        \\ \notag
        &\quad+\notag
            \frac{\nu}{20}\norm{\bw}^2_{V^{s+\alpha}}
        + C_\nu
            \norm{\bu}_{V^\alpha}^2
            \norm{\bw}_{V^s}^{2\theta}
            \norm{\bw}_{V^{s+\alpha}}^{2(1-\theta)}
        \\ \notag
    &\leq
            \frac{\nu}{20}\norm{\bw}_{V^{s+\alpha}}^2
        + C
            \norm{\bu}^{\frac{2}{\theta}}_{V^s}
            \norm{\bw}^{2}_{V^s}
        +
            \frac{\nu}{20}\norm{\bw}^{2}_{V^{s+\alpha}}
        \\ \notag
        &\quad+\notag
            \frac{\nu}{20}\norm{\bw}^2_{V^{s+\alpha}}
        + C
            \norm{\bu}_{V^\alpha}^{\frac{2}{\theta}}
            \norm{\bw}_{V^s}^{2}
        +
            \frac{\nu}{20}\norm{\bw}_{V^{s+\alpha}}^{2}
        \\ \notag
    &\leq 
            \frac{4\nu}{20}\norm{\bw}^2_{V^{s+\alpha}}
        + C
            \left( 
                    \norm{\bu}_{V^s}^{\frac{2}{\theta}} 
                + 
                    \norm{\bu}_{V^\alpha}^{\frac{2}{\theta}}
            \right)
            \norm{\bw}_{V^s}^2,
\end{align}
}


For the second nonlinear term,
{\allowdisplaybreaks
\begin{align}
\abs{b(\bw,\bu,A^s\bw)} 
    &= 
        \abs{(\Lambda^s B(\bw,\bu),\Lambda^s\bw)}
    \\ \notag
    &\leq 
        \norm{\Lambda^sB(\bw,\bu)}_{L^{4/3}}
        \norm{\Lambda^s\bw}_{L^4}
    \\ \notag
    &\leq C
        \norm{\Lambda^s\bw}_{L^4}^2
        \norm{\nabla \bu}_{L^2}\\ \notag
     &\quad+C \notag
        \norm{\bw}_{L^4}
        \norm{\Lambda^s\nabla\bu}_{L^2}
        \norm{\Lambda^s\bw}_{L^4}
    \\ \notag
    &= C
            \norm{A^{s/2}\bw}_{L^4}^2
            \norm{\nabla \bu}_{L^2}\\ \notag
        &\quad+C \notag
            \norm{\bw}_{L^4}
            \norm{A^{s/2}\nabla\bu}_{L^2}
            \norm{A^{s/2}\bw}_{L^4}
    \\ \notag
    &\leq C
            \norm{\bw}_{V^s}^{2\theta}
            \norm{\bw}_{V^{s+\alpha}}^{2(1-\theta)}
            \norm{\bu}_{V}\\ \notag
        &\quad+C \notag
            \norm{\bw}_{V^\alpha}
            \norm{\bu}_{V^{s+1}}
            \norm{\bw}_{V^s}^{\theta}
            \norm{\bw}_{V^{s+\alpha}}^{1-\theta}
        \\ \notag
    &\leq 
            \frac{\nu}{20}\norm{\bw}_{V^{s+\alpha}}^{2}
        + C
            \norm{\bu}_{V}^{\frac{1}{\theta}}
            \norm{\bw}_{V^s}^{2}
        \\ \notag
        &\quad+\notag
            \frac{\nu}{20}\norm{\bw}_{V^{s+\alpha}}^{2}
        + C
            \norm{\bu}_{V^{s+1}}^2
            \norm{\bw}_{V^s}^{2\theta}
            \norm{\bw}_{V^{s+\alpha}}^{2(1-\theta)}
        \\ \notag
    &\leq 
            \frac{\nu}{20}\norm{\bw}_{V^{s+\alpha}}^{2}
        + C
            \norm{\bu}_{V}^{\frac{1}{\theta}}
            \norm{\bw}_{V^s}^{2}
        \\ \notag
        &\quad+ \notag
            \frac{\nu}{20}\norm{\bw}_{V^{s+\alpha}}^{2}
        + C
            \norm{\bu}_{V^{s+1}}^{\frac{2}{\theta}}
            \norm{\bw}_{V^s}^{2}
        + 
            \frac{\nu}{20}\norm{\bw}_{V^{s+\alpha}}^{2}
        \\ \notag
    &\leq 
            \frac{3\nu}{20}\norm{\bw}_{V^{s+\alpha}}^2
        + C
             \left( 
                    \norm{\bu}^{\frac{1}{\theta}}_{V} 
                + 
                    \norm{\bu}_{V^{s+1}}^{\frac{2}{\theta}} 
            \right)
            \norm{\bw}_{V^s}^2.
\end{align}
}
For the last nonlinear term,
{\allowdisplaybreaks
\begin{align}
    \abs{b(\bw,\bw,A^s\bw)} &= \abs{(\Lambda^sB(\bw,\bw),\Lambda^s\bw)}\\ \notag
    &\leq C\norm{\Lambda^sB(\bw,\bw)}_{L^{4/3}}\norm{\Lambda^s\bw}_{L^4}\\ \notag
    &\leq C
            \norm{\Lambda^s\bw}_{L^4}^2
            \norm{\nabla \bw}_{L^2}
        \\ \notag
    &\quad+ C\notag
            \norm{\bw}_{L^4}
            \norm{\Lambda^s\nabla \bw}_{L^2}
            \norm{\Lambda^s\bw}_{L^4}
        \\ \notag
    &= C
            \norm{A^{s/2}\bw}_{L^4}^2
            \norm{\nabla \bw}_{L^2}
        \\ \notag
    &\quad+ C\notag
            \norm{\bw}_{L^4}
            \norm{A^{s/2}\nabla \bw}_{L^2}
            \norm{A^{s/2}\bw}_{L^4}
        \\ \notag
    &\leq C
            \norm{\bw}_{V^s}^{2\theta}
            \norm{\bw}_{V^{s+\alpha}}^{2(1-\theta)}
            \norm{\bw}_{V}
        \\ \notag
        &\quad+C \notag
            \norm{\bw}_{V^\alpha}
            \norm{\bw}_{V^{s+1}}
            \norm{\bw}_{V^s}^{\theta}
            \norm{\bw}_{V^{s+\alpha}}^{1-\theta}
        \\ \notag
    &\leq
            \frac{\nu}{20}\norm{\bw}_{V^{s+\alpha}}^2
        +C_{\alpha,\nu}
            \norm{\bw}_{V}^{\frac{1}{\theta}}
            \norm{\bw}_{V^s}^{2}
        \\ \notag
        &\quad+ \notag
            \frac{\nu}{20}\norm{\bw}_{V^{s+\alpha}}^2
        +C_{\alpha,\nu}
            \norm{\bw}_{V^\alpha}^2
            \norm{\bw}_{V^s}^{2\theta}
            \norm{\bw}_{V^{s+\alpha}}^{2(1-\theta)}
        \\ \notag
    &\leq
            \frac{\nu}{20}\norm{\bw}_{V^{s+\alpha}}^2
        +C_{\alpha,\nu}
            \norm{\bw}_{V}^{\frac{1}{\theta}}
            \norm{\bw}_{V^s}^{2}
        \\ \notag
        &\quad+ \notag
            \frac{\nu}{20}\norm{\bw}_{V^{s+\alpha}}^2
        + C_{\alpha,\nu}
            \norm{\bw}_{V^\alpha}^{\frac{2}{\theta}}
            \norm{\bw}_{V^s}^{2}
        +
            \frac{\nu}{20}\norm{\bw}_{V^{s+\alpha}}^{2}
        \\ \notag
    &\leq 
        \frac{3\nu}{20}\norm{\bw}_{V^{s+\alpha}}^2 
    + C_{\alpha,\nu}
        \left( 
                \norm{\bw}^{\frac{1}{\theta}}_{V} 
            + 
                \norm{\bw}_{V^{\alpha}}^{\frac{2}{\theta}} 
        \right)
        \norm{\bw}_{V^s}^2.
\end{align}
}

Putting these estimates together we obtain:

\begin{align}
    \abs{b(\bu,\bw,A^s\bw)} + \abs{b(\bw,\bu,A^s\bw)} + \abs{b(\bw,\bw,A^s\bw)}
    \leq 
        \frac{10\nu}{20}\norm{\bw}_{V^{s+\alpha}}^2
    \\ \notag 
    +C 
        \left(
                \norm{\bu}_{V^s}^{\frac{2}{\theta}}
            +
                \norm{\bu}_{V^\alpha}^{\frac{2}{\theta}}
            +
                \norm{\bu}_{V^{1}}^{\frac{1}{\theta}}
            +
                \norm{\bu}_{V^{s+1}}^{\frac{2}{\theta}}
            +
                \norm{\bw}_{V}^{\frac{2}{\theta}}
            +
                \norm{\bw}_{V^{\alpha}}^{\frac{2}{\theta}}
        \right)
    \norm{\bw}_{V^{s}}^2.
    \notag 
\end{align}

{\allowdisplaybreaks
We now estimate the interpolant term as
\begin{align}
    -\mu\left\langle \bw, A^s\bw \right\rangle
    &=
    \mu
        \left\langle \bw - I_h(\bw), A^s \bw \right\rangle
    - \mu
        \left\langle \bw, A^s \bw \right\rangle
    \\ \notag
    &\leq
    \mu
        \abs{
            \left\langle A^{s/2}\left(\bw - I_h(\bw)\right), A^{s/2} \bw \right\rangle
        }
    -\mu
        \norm{\bw}_{V^s}^2
    \\ \notag
    &\leq
    \mu
        \norm{ A^{s/2}\left(\bw - I_h(\bw)\right)}_{L^2}
        \norm{A^{s/2} \bw}_{L^2}
    -\mu
        \norm{\bw}_{V^s}^2    
    \\ \notag
    &\leq
    \mu
        \norm{\bw - I_h(\bw)}_{V^s}
        \norm{\bw}_{V^s}
    -\mu
        \norm{\bw}_{V^s}^2  
    \\ \notag
    &\leq
    \frac{\mu}{2}
        \norm{\bw - I_h(\bw)}_{V^s}^2
    +\frac{\mu}{2}
        \norm{\bw}_{V^s}^2
    -\mu
        \norm{\bw}_{V^s}^2    
    \\ \notag
    &\leq
    \frac{\mu h^2 c_{0,\alpha}}{2}
        \norm{\bw}_{V^{s+\alpha}}^2
    -\frac{\mu}{2}
        \norm{\bw}_{V^s}
    \\ \notag
    &\leq
    \frac{\nu}{2}
        \norm{\bw}_{V^{s+\alpha}}^2
    -\frac{\mu}{2}
        \norm{\bw}_{V^s}.
\end{align}
}

Putting these estimates together
we obtain,

\begin{align}
\quad&\notag
        \dv{~}{t}\norm{\bw}^2_{V^s} 
    \\+&\notag
        \left(
                \mu
            -C 
            \left(
                    \norm{\bu}_{V^s}^{\frac{2}{\theta}}
                +
                    \norm{\bu}_{V^\alpha}^{\frac{2}{\theta}}
                +
                    \norm{\bu}_{V^{1}}^{\frac{1}{\theta}}
                +
                    \norm{\bu}_{V^{s+1}}^{\frac{2}{\theta}}
                +
                    \norm{\bw}_{V}^{\frac{2}{\theta}}
                +
                    \norm{\bw}_{V^{\alpha}}^{\frac{2}{\theta}}
            \right)    
        \right)
        \norm{\bw}_{V^s}^2 
    \leq 0.
\end{align}
Thus, for $ \mu
    \geq C_{\alpha,\nu}\left(4\sigma_\alpha^{2/\theta} + \sigma_{1}^{1/\theta} + 2\sigma_{1}^{2/\theta} + \sigma_{\alpha+1}\right),  $
we have exponential convergence of $\bv$ to $\bu$ in $V^s$. We know from before that $\bu,\bv \in L^\infty(0,T;V^{2\alpha})$, and so \newline $\bw\in L^\infty(0,T;V^{2\alpha})$. Thus, setting $s=\alpha$ we obtain the result. 
\end{proof}

\section{Conclusion}
\label{conclusion}

\noindent
In this study, we applied the AOT algorithm to the Navier--Stokes equations with fractional diffusion.
We demonstrated the global \hyphenation{well-posedness} for the assimilation equations and proved that their solutions converge to the reference solution in the  $H$ and $V^\alpha$ norms exponentially fast in time, and for spatial dimensions $d$, $2\leq d\leq 8$.  
We note that in the 3D case, these results apply with diffusion exponent of at least $5/4$, which is known as the Lions exponent, for which global well-posedness of the 3D fractional Navier--Stokes system was shown in \cite{Lions_1959}. 


The sequel to this study \cite{Carlson_Larios_Victor_2023_2D_fractional_AOT} will be focused on lower-order diffusion (i.e., ``hypo-diffusion'') in the 2D case.
More specifically, we are interested in studying the behavior of the system as we let the parameter $\alpha$ approach $0$. 
This case is of particular interest as the 2D scenario is known to be globally well-posed for any choice of $\alpha\geq0$ (and also $\nu=0$) with sufficiently regular initial data. 

\appendix
\section{Uniform Bounds on Solutions}
\begin{theorem}\label{thm - unif bounds}
Let $\mu\geq0$, $\nu>0$, $T>0$, $2\leq d\leq 8$, and $\alpha\geq \frac{d}{4}+\frac12$.  Given $\bbf\in V^{\alpha}$, $\bu_0\in V^\alpha$, and $\bv_0\in V^\alpha$, let $\bu$ and $\bv$, be classical solutions of \cref{eq: NSE functional,eq: NSE AOT functional}, respectively, on $[0,T]$ such that $\norm{\bv_0}_{v^\alpha}\leq K$.  Then there exists a $\sigma_s>0$ such that \begin{align}
        \norm{\bv(t)}_{V^s} \leq \sigma_s
    \end{align}
    for $s = 0,1,\alpha$ and for all times $t\in [0,T]$. Moreover, $\sigma_0$ does not depend on time.
\end{theorem}

\begin{proof}
To prove this, we must first establish a uniform bound in $V$. To do this, we take the inner product of \cref{eq: NSE AOT functional} with $\bv$, yielding the following equations:

\begin{align*}
\frac{1}{2}\dv{~}{t}\norm{\bv}^2_{L^2} + \nu \norm{\bv}_{V^\alpha}^2 + b(\bv,\bv,\bv) = (\vect{f}, \bv) + \mu (I_h(\bu),\bv) - \mu (I_h(\bv),\bv).    
\end{align*}
Let us now denote $\vect{g}:= \vect{f} + \mu I_h(\bu)$, and note that due to the antisymmetric property of the trilinear form $b$, the equations simplify to the following:
\begin{align}
\frac{1}{2}\dv{~}{t}\norm{\bv}^2_{L^2} + \nu \norm{\bv}_{V^\alpha}^2 = (\vect{g}, \bv)- \mu (I_h(\bv),\bv).    
\end{align}
We note that, as $\bu$ is a classical solution, and due to the requirements on $I_h$ we have that $\norm{\vect{g}}_{L^\infty(0,T;L^2)} \leq M$ for some constant $M>0$.
We now estimate each of these terms as follows:
\begin{align}
    (\vect{g}, \bv) &\leq \frac{\norm{\vect{g}}^2_{L^2}}{\mu} + \frac{\mu}{2}\norm{\bv}^2_{L^2}
\end{align}
\begin{align}
    - \mu (I_h(\bv),\bv)  &= \mu (\bv - I_h(\bv), \bv) - \mu(\bv,\bv)\\
                            &\leq \mu \norm{\bv - I_h(\bv)}_{L^2}\norm{\bv}_{L^2} - \mu \norm{\bv}^2_{L^2}\notag \\
                            &\leq \frac{\mu}{2} \norm{\bv - I_h(\bv)}_{L^2}^2 + \frac{\mu}{2}\norm{\bv}^2_{L^2} - \mu \norm{\bv}^2_{L^2}\notag \\
                            &\leq \frac{\mu c_{0,\alpha} h^2}{2} \norm{A^{\alpha/2}\bv}_{L^2}^2  - \frac{\mu}{2}\norm{\bv}^2_{L^2}\notag \\\notag
                            &\leq \frac{\nu}{2} \norm{\bv}_{V^\alpha}^2  - \frac{\mu}{2}\norm{\bv}^2_{L^2}.
\end{align}

Putting these estimates together, we obtain
\begin{align}
    \dv{}{t}\norm{\bv}_{L^2}^2 + \nu \norm{\bv}_{V^\alpha}^2 &\leq \frac{\norm{\vect{g}}_{L^2}^2}{\mu}
    \leq \frac{M}{\mu}.
\end{align}
Using Poincar\'e's inequality, it follows that
\begin{align}
    \dv{~}{t}\norm{\bv}_{L^2}^2 + \nu \lambda_1^\alpha \norm{\bv}_{L^2}^2 &\leq \frac{M}{\mu}.
\end{align}

Now, using Gr\"onwall's inequality we obtain
\begin{align}
    \norm{\bv(t)}_{L^2}^2\leq \norm{\bv_0}^2_{L^2}e^{-\nu \lambda_1^\alpha t} + \frac{M}{\mu\nu\lambda_1^\alpha}(1-e^{-\nu \lambda_1^\alpha t})\leq K + \frac{M}{\mu\nu\lambda_1}:= \sigma_0.
\end{align}

We note that one could pick $t_0(\norm{\bv_0}_{L^2}) = \max( -\frac{1}{\nu\lambda_1^\alpha}\ln(\frac{M}{\mu\nu\lambda_1^\alpha\norm{\bv_0}_V^2}) ,0)$,
which will yield that for all $t\geq t_0$
\begin{align}
    \norm{\bv(t)}_{L^2}^2 \leq \frac{2M}{\mu\nu\lambda_1^\alpha}.
\end{align}
Thus we have proved the existence of an absorbing ball in $H$.
Additionally, we also have that
\begin{align}\label{est: int V alpha}
    \int_{0}^T\norm{\bv}^2_{V}ds &\leq T\frac{\sigma_0^2}{\nu} + \frac{K}{\nu}:=\rho_1.
\end{align}

We now move to prove the existence of an absorbing ball in $V$. To do this, we instead take  the inner product of \cref{eq: NSE AOT functional} with $A\bv$. This yields the following equations:
\begin{align}
\frac{1}{2}\dv{~}{t}\norm{\bv}_{V}^2 + \nu\norm{\bv}^2_{V^{\alpha+1}}&= -b(\bv,\bv,A\bv) + \left\langle\vect{g}, A\bv\right\rangle 
- \mu \left\langle I_h(\bv), A \bv \right\rangle.
\end{align}

We will now utilize two different estimates for the nonlinear term. For $\frac{d}{4}+\frac{1}{2}\leq \alpha < \frac{d}{2}+1$, integrating by parts and using \eqref{symm2}, we obtain
\begin{align}
b(\bv,\bv,A\bv) &= \sum_{i,j,k=1}^{d}\int_\Omega \pdv{v_i}{x_k}\pdv{v_j}{x_i} \pdv{v_j}{x_k}  \,d\bx
\leq C\norm{\nabla \bv}_{L^3}^3\\ \notag
     &\leq C\norm{\nabla \bv}^{3b_1}_{L^2}\norm{A^{\frac{\alpha}{2}}\bv}^{3b_2}_{L^2}\norm{A^{\frac{\alpha + 1}{2}}\bv}^{3b_3}_{L^2}\\ \notag
     &\leq \frac{\nu}{4}\norm{\bv}^{2}_{V^{\alpha+1}} + C\norm{ \bv}^{2}_{V}\norm{\bv}^{2\frac{b_2}{b_1}}_{V^\alpha}.
\end{align}
Here $(b_1,b_2,b_3) = (1 - \gamma, \frac{1}{3}, \gamma - \frac{1}{3})$, where $\gamma = \frac{1}{3\alpha}(1 + \frac{d}{2})$. 

For larger values of $\alpha$, we will instead utilize Agmon's inequality to estimate the nonlinear term as follows
\begin{align}
    \abs{b(\bv,\bv,A \bv)}& \leq C\norm{\bv}_{L^\infty}\norm{\bv}_{V} \norm{\bv}_{V^2}\\ \notag
    &\leq C \norm{\bv}_{V^{\frac{d}{2} - 1}}^{1/2}\norm{\bv}_{V^{\frac{d}{2}+1}}^{1/2}\norm{\bv}_{V}\norm{\bv}_{V^2}\\ \notag
    &\leq  \Calpha  \norm{\bv}_{V^\alpha}\norm{\bv}_{V}\norm{\bv}_{V^{\alpha+1}}\\ \notag
   &\leq   C_{\alpha,\nu} \norm{\bv}_{V^\alpha}^2\norm{\bv}_{V}^2 + \frac{\nu}{4}\norm{\bv}_{V^{\alpha + 1}}^2.
\end{align}

We estimate the forcing term as follows
\begin{align}
    (\vect{g},A\bv) &\leq \frac{\norm{\vect{g}}_{V}^2}{\mu} + \frac{\mu}{2}\norm{\bv}_{V}^2.
\end{align}
We now estimate the remaining term as
\begin{align}
    -\mu(I_h(\bv),A\bv) &= \mu(\bv - I_h(\bv), A^s\bv) - \mu \norm{\bv}_{V}^2\\\notag&= \mu(A^\frac{1}{2}\left(\bv - I_h(\bv)\right), A^\frac{1}{2}\bv) - \mu \norm{\bv}_{V}^2\\\notag&\leq \mu\norm{\bv - I_h(\bv)}_{V}\norm{\bv}_{V} - \mu \norm{\bv}_{V}^2\\\notag&\leq \frac{\mu c_{0,\alpha}h^2}{4}\norm{\bv - I_h(\bv)}_{V}^2 + \frac{\mu}{2}\norm{\bv}^2_{V} - \mu \norm{\bv}_{V}^2\\\notag&\leq \frac{\mu c_{0,\alpha}h^2C_{\alpha}}{2}\norm{\bv}_{V}^2 + \frac{\mu}{2}\norm{\bv}^2_{V} - \mu \norm{\bv}_{V}^2\\\notag&\leq \frac{\nu}{4}\norm{\bv}_{V^{\alpha+1}}^2 - \frac{\mu}{2} \norm{\bv}_{V}^2.
\end{align}

Putting these estimates together we obtain
\begin{align}\label{ineq: v ball}
    \dv{}{t}\norm{\bv}_{V}^2 + \nu\norm{\bv}_{V^{\alpha+1}}^2 \leq \frac{\norm{\vect{g}}_{V}^2}{\mu} + C\norm{\bv}_{V^\alpha}^{2a}\norm{\bv}_{V}^2,
\end{align}
where 
\begin{align}b := \begin{cases}
    \frac{b_1}{b_2}, & \frac{d}{4}+\frac{1}{2}\leq \alpha < \frac{d}{2}+1,\\\notag1, & \alpha \geq \frac{d}{2}+1.
\end{cases}.
\end{align}
Note that $\norm{\vect{g}}_{V}^2 \leq M$, for $\vect{f}\in V$. Using Poincar\'e's inequality we obtain
\begin{align}
    \dv{}{t}\norm{\bv}_{V}^2 + \nu\lambda_1^\alpha\norm{\bv}_{V}^2 \leq \frac{M}{\mu} + C\norm{\bv}_{V^\alpha}^{2b}\norm{\bv}_{V}^2.
\end{align}
Hence, from Gr\"onwall's inequality, it follows that
\begin{align}
    \norm{\bv(t)}_{V}^2\leq \underbrace{\norm{\bv(t_0)}_{V}^2 e^{\left(\left(\rho_1\right)^b-\nu\lambda_1^\alpha\right)t}+ \frac{M}{\nu}\frac{1}{\left(\rho_1\right)^b - \nu\lambda_1^\alpha}\left(1 - e^{\left(\left(\rho_1\right)^b - \nu\lambda_1^\alpha\right)t}\right)}_{:=\sigma_1^2} .
\end{align}

We note that here we are using H\"older's inequality to obtain
\begin{align}
    \int_{0}^t\norm{\bv}_{V^\alpha}^{2b} \leq (t-0)^{1-b}\left(\int_{0}^t\norm{\bv}_{V^\alpha}^2dt\right)^b \leq t\left(\rho_1\right)^b.
\end{align}
This is permissible in the regime of $\frac{d}{4} + \frac{1}{2}\leq \alpha < \frac{d}{2} + 1$, since $\frac1b> 1$. Outside this regime, the use of H\"older's inequality is unnecessary, as we can use estimate \cref{est: int V alpha} directly.

Similarly to before, we have that 
\begin{align}\label{ineq: V alpha + 1 ball}
    \int_{0}^T\norm{\bv}_{V^{\alpha+1}}^2 \leq \rho_{\alpha + 1},
\end{align}
where the explicit value of $\rho_{\alpha +1}$ can be obtained by integrating \cref{ineq: v ball} and using the previous estimates on $\norm{\bv(t)}^2_{V}$.

We note here that in 2D periodic setting the nonlinear term vanishes, which allows one to achieve better estimates and obtain an absorbing ball in $V$ with radius defined in terms of the Grashof number. While in $H$ it is straightforward to show the absorbing ball, we require only a uniform bound on $\norm{\bv(t)}^2_{V}$. 
Thus, there is a dependence on $\norm{\bv_0}_{V}$ in $\sigma_1$ and $\rho_\alpha$.

Finally, to achieve a uniform bound on $\norm{\bv}_{V^\alpha}$ we take the inner product of \cref{eq: NSE AOT functional} with $A^\alpha \bv$. This yields the following set of equations:
\begin{align}
\frac{1}{2}\dv{}{t}\norm{\bv}^2_{V^\alpha} + \nu\norm{\bv}^2_{V^{2\alpha}} = (\vect{g},A^{\alpha}\bv) - b(\bv,\bv, A^\alpha\bv) - \mu(I_h(\bv),\bv).
\end{align}
Similarly to before we estimate each of the terms on the right hand side,
\begin{align}
    \left(\vect{g}, A^\alpha \bv\right) &\leq \frac{\norm{\vect{g}}^2_{V^\alpha}}{\mu} + \frac{\mu}{2}\norm{\bv}_{V^\alpha}^2,
\end{align}
\begin{align}
    -\mu (I_h(\bv), A^\alpha \bv) &\leq \frac{\nu}{2}\norm{\bv}^2_{V^{2\alpha}} - \frac{\mu}{2}\norm{\bv}_{V^\alpha}^2,
\end{align}
and 
\begin{align}
    \abs{b(\bv,\bv,A^\alpha \bv)} & = \abs{(B(\bv,\bv), A^\alpha \bv)}\\\notag&=  \abs{(\Lambda^{\alpha}B(\bv,\bv), \Lambda^{\alpha} \bv)}\\\notag&\leq \norm{\Lambda^{\alpha}B(\bv,\bv)}_{L^{4/3}}\norm{ \Lambda^{\alpha} \bv}_{L^4}\\\notag&\leq C
        \norm{A^{\alpha/2}\bv}_{L^4}^2
        \norm{\bv}_{V} 
    + C
        \norm{\bv}_{L^4}
        \norm{A^{\alpha/2}\grad\bv}_{L^2}
        \norm{A^{\alpha/2}\bv}^2_{L^4}\\\notag&\leq C
        \norm{\bv}_{V^\alpha}^{2\theta}
        \norm{\bv}_{V^{2\alpha}}^{2(1-\theta)}
        \norm{\bv}_{V}
    + C
        \norm{\bv}_{V^\alpha}
        \norm{\bv}_{V^{\alpha+1}}
        \norm{\bv}_{V^{2\alpha}}\\\notag&\leq \frac{\nu}{2}    
        \norm{\bv}_{V^{2\alpha}}^{2}
    + C_{\alpha,\nu}
        \norm{\bv}_{V^\alpha}^{2}
        \norm{\bv}_{V}^{\frac{1}{\theta}}
    + C
        \norm{\bv}_{V^\alpha}^2
        \norm{\bv}_{V^{\alpha+1}}^2.\notag
\end{align}
Here $\theta = 1 - \frac{d}{4\alpha}$.
Note that above we have utilized the Kato-Ponce inequality, \cref{ineq:Kato-Ponce} in our estimate of the nonlinear term.

Finally, putting all of this together, we obtain the following inequality
\begin{align}
    \dv{}{t}\norm{\bv}_{V^{\alpha}}^2 & \leq \frac{\norm{\vect{g}}_{V^\alpha}^2}{\mu}     + C_{\alpha,\nu}\norm{\bv}_{V}^{\frac{1}{\theta}}\norm{\bv}_{V^\alpha}^{2}+ C\norm{\bv}_{V^\alpha}^2\norm{\bv}_{V^{\alpha+1}}^2\\\notag&\leq \frac{M}{\mu}     + C_{\alpha,\nu}\sigma_1^{\frac{1}{\theta}}\norm{\bv}_{V^\alpha}^{2}+ C\norm{\bv}_{V^\alpha}^2\norm{\bv}_{V^{\alpha+1}}^2.
\end{align}
Utilizing  Gr\"onwall 's inequality and \cref{ineq: V alpha + 1 ball} we obtain a bound $\norm{\bv}_{L^\infty(0,T;V^{\alpha})}\leq \sigma_{\alpha}$.

We note that this bound holds for any choice of initial data $\bv_0 \in V^\alpha$ with $\norm{\bv_0}_{V^\alpha}^2 \leq K$, as there is a dependence of $\sigma_\alpha$ on $K$.
\end{proof}


We note that the above proof can be readily adapted to provide a similar estimate for $\bu$, with $\norm{\bu}^2_{V^\alpha} \leq \tilde\sigma_\alpha$. In this case the resultant $\tilde\sigma_\alpha$ has no dependence on $\mu$ whatsoever. We  take $\sigma_\alpha := \max(\sigma_\alpha, \tilde\sigma_\alpha)$ to achieve a uniform bound on both $\bu$ and $\bv$. Alternatively, we note that the uniform bounds on $\bu$ follow trivially from the regularity results in \cref{thm - classical Existence of u}. The uniform bounds on $\bv$ also follow directly from this theorem, but the estimates above have a reduced dependence on $\mu$.

\section*{Acknowledgments}
\noindent
Author A.L. would like to thank the Isaac Newton Institute for Mathematical Sciences, Cambridge, for support and warm hospitality during the programme ``Mathematical aspects of turbulence: where do we stand?'' where work on this manuscript was undertaken. This work was supported by EPSRC grant no. EP/R014604/1. 
The research of A.L. was supported in part by NSF grants DMS-2206762 and CMMI-1953346, and USGS grant G23AS00157 number GRANT13798170.  
The research of C.V. was supported in part by the NSF GRFP grant DMS-1610400.

\begin{small}

\end{small}

\end{document}